\documentclass[10pt,reqno]{amsart}

\usepackage[utf8]{inputenc}
\usepackage[english]{babel}
\usepackage[dvipsnames]{xcolor}
\usepackage{amsmath, amssymb, amsthm}
\usepackage{hyperref}
\usepackage{mathrsfs}
\usepackage{tikz-cd}
\usepackage{enumerate}
\usepackage{esint}
\usepackage{xcolor}
\usepackage{mathtools}
\usepackage{bm}
\usepackage{comment}

\usepackage{amsopn}
\DeclareMathOperator{\diag}{diag}

\usepackage{xparse}%

\usepackage{color}
\usepackage{enumerate}

\newtheorem{theorem}{Theorem}[section]
\newtheorem{lemma}[theorem]{Lemma}

\theoremstyle{remark}
\newtheorem{remark}[theorem]{\bf{Remark}}

\theoremstyle{definition}
\newtheorem{assumption}[theorem]{Assumption}

\newtheorem{definition}[theorem]{Definition}

\newcommand\cbrk{\text{$]$\kern-.15em$]$}}
\newcommand\opar{\text{\,\raise.2ex\hbox{${\scriptstyle
|}$}\kern-.34em$($}}
\newcommand\cpar{\text{$)$\kern-.34em\raise.2ex\hbox{${\scriptstyle |}$}}\,}

\newcommand{\aint}{-\hspace{-0.38cm}\int}

\newcommand{\mysection}[1]{\section{#1}
\setcounter{equation}{0}}

\begin{document}

\title[An $L_q(L_p)$-theory for parabolic equations with anisotropic non-local operators]
{A regularity theory for parabolic equations with anisotropic non-local operators in $L_{q}(L_{p})$ spaces}

\thanks{The first author has been supported by Miwon Du-Myeong Fellowship. The second author has been supported by BK21 SNU Mathematical Sciences Division. The third author has been supported by the Samsung Sciences \& Technology Foundation(SSTF)’s grants (No. SSTF-BA1401-51)} 

\author{Jae-Hwan Choi}
\address{Department of Mathematical Sciences, Korea Advanced Institute of Science and Technology, 291 Daehak-ro, Yuseong-gu, Daejeon, 34141, Republic of Korea} 
\email{jaehwanchoi@kaist.ac.kr}

\author{Jaehoon Kang}
\address{Department of Mathematical Sciences, Seoul National University, Building 27, 1 Gwanak-ro, Gwanak-gu
Seoul 08826, Republic of Korea} 
\email{jhnkang@snu.ac.kr}

\author{Daehan Park}
\address{School of Mathematics, Korea Institute for Advanced Study, 85 Hoegi-ro Dongdaemun-gu, Seoul 02455, Republic of Korea.} 
\email{daehanpark@kias.re.kr}



\begin{abstract}
In this paper, we present an $L_q(L_p)$-regularity theory for parabolic equations of the form:
$$
\partial_t u(t,x)=\mathcal{L}^{\vec{a},\vec{b}}(t)u(t,x)+f(t,x),\quad u(0,x)=0.
$$
Here, $\mathcal{L}^{\vec{a},\vec{b}}(t)$ represents anisotropic non-local operators encompassing the singular anisotropic fractional Laplacian with measurable coefficients:
$$
\mathcal{L}^{\vec{a},\vec{0}}(t)u(x)=\sum_{i=1}^{d} \int_{\mathbb{R}}\left( u(x^{1},\dots,x^{i-1},x^{i}+y^{i},x^{i+1},\dots,x^{d}) - u(x) \right) \frac{a_{i}(t,y^{i})}{|y^{i}|^{1+\alpha_{i}}} \mathrm{d}y^{i} .
$$
To address the anisotropy of the operator, we employ a probabilistic representation of the solution and Calder\'on-Zygmund theory.
As applications of our results, we demonstrate the solvability of elliptic equations with anisotropic non-local operators and parabolic equations with isotropic non-local operators.
\end{abstract}

\maketitle

\tableofcontents 

\mysection{Introduction}
The study of partial differential equations (PDEs) has yielded extensive research on parabolic equations, given their ability to provide mathematical descriptions of natural and artificial phenomena. The most basic form of parabolic equations is the heat equation, represented as $\partial_{t}u = \Delta u$, which describes the diffusion of free particles. An interesting observation is that the fundamental solution to the heat equation is the transition density of Brownian motion $B_{t}$, where the Laplacian $\Delta$ acts as the infinitesimal generator of Brownian motion;
\begin{align*}
\lim_{t\downarrow 0} \frac{\mathbb{E}u(x+B_{t}) - u(x)}{t} = \Delta u(x).
\end{align*}
This observation has inspired researchers to investigate the relationship between stochastic processes and PDEs. In particular, the \emph{subordinate Brownian motion} (SBM) $X_{t}:=B_{S_{t}}$ and the corresponding PDE have been thoroughly studied, because SBM can describe not only normal diffusion but also pure jump diffusion. Here, $S_{t}$ is a real-valued increasing L\'evy process that satisfies $\mathbb{E}[\mathrm{e}^{-\lambda S_{t}}] = \mathrm{e}^{-t\phi(\lambda)}$, where $\phi$ denotes a Bernstein function. For more details, see Section \ref{23.04.19.21.02}.

There are numerous fields, including physics, engineering, and finance, where SBM is used for relevant modeling (see \textit{e.g.} \cite{C03option,fogedby1994,S03sub}). One key characteristic of SBM is its radial symmetry, which makes it particularly suitable for modeling isotropic diffusion processes like \cite{H08MRI}.
However, SBM is not applicable for the analysis or description of anisotropic diffusion models, which are frequently investigated using advanced magnetic resonance imaging (MRI) techniques such as diffusion tensor imaging (DTI) and diffusion kurtosis imaging (DKI).
In \cite{de2011anisotropic}, for instance, the authors examined DTI theory by employing a three-dimensional Lévy process $X_t:=(X^1_t,X^2_t,X_t^3)$ with independent one-dimensional $\alpha_1$-stable, $\alpha_2$-stable, $\alpha_3$-stable processes $X^1_t$, $X^2_t$, $X_t^3$, respectively, where $\alpha_1,\alpha_2,\alpha_3 \in (0,2)$.
While it is evident that each component process $X^1$, $X^2$, $X^3$ are SBMs, the process $X$ itself is not an SBM, and its radial symmetry is consequently broken.
To help readers understand, we provide detailed elucidations of this phenomenon from two perspectives, starting with the probabilistic viewpoint.
Given the independence of each component process, we can establish the following relationship:
$$
\sum_{0<s<t}|X_s^i-X_{s-}^i||X_s^j-X_{s-}^j|=0\quad a.s,\quad i,j=1,2,3,\,i\neq j.
$$
That is, there are no simultaneous jumps between component processes.
This indicates that $X$ only moves parallel to the coordinate axes.
Furthermore, given two open sets
$$
Q_1:=(-a_1,a_1)\times(-a_2,a_2)\times(-a_3,a_3),\quad Q_2=:(-a_1,a_1)\times(-a_3,a_3)\times(-a_2,a_2),
$$
which are identical up to certain rotation, we also observe that
\begin{align*}
    \mathbb{P}(X_t\in Q_1)&=\prod_{i=1}^{3}\mathbb{P}(X^i_t\in (-a_i,a_i))\\
    &\neq \mathbb{P}(X_t^1\in (-a_1,a_1))\mathbb{P}(X_t^2\in (-a_3,a_3))\mathbb{P}(X_t^3\in (-a_2,a_2))\\
    &= \mathbb{P}(X_t\in Q_2).
\end{align*}
This implies that the probability distribution of $X_t$ does not exhibit radial symmetry. Another perspective is based on analysis and PDEs.
The infinitesimal generator of $X$ is a singular anisotropic fractional Laplacian $\mathcal{A}$, given by
\begin{align}
\label{23.02.21.17.23}
    \mathcal{A}u(x)=\int_{\mathbb{R}^{3}} \left( u(x+y)-u(x) -\nabla u(x) \cdot y \mathbf{1}_{|y|\leq1} \right) \pi(\mathrm{d}y),
\end{align}
where $y = (y_1,y_2,y_3) \in \mathbb{R}^3$, and $\pi(\mathrm{d}y)$ is defined by
\begin{align*}
    \pi(\mathrm{d}y) &:= c_{1}|y_{1}|^{-1-\alpha_{1}}\mathrm{d}y_{1}\,\epsilon_{0}(\mathrm{d}y_{2},\mathrm{d}y_{3}) + c_{2}|y_{2}|^{-1-\alpha_{2}}\mathrm{d}y_{2}\,\epsilon_{0}(\mathrm{d}y_{1},\mathrm{d}y_{3})\\
    &\quad+c_{3}|y_{3}|^{-1-\alpha_{3}}\mathrm{d}y_{3}\,\epsilon_{0}(\mathrm{d}y_{1},\mathrm{d}y_{2}).
\end{align*}
Here $c_{1},c_{2},c_{3}$ are positive constants, and $\epsilon_0$ is the Dirac measure on $\mathbb{R}^2$ centered at the origin. The operator $\mathcal{A}$ in \eqref{23.02.21.17.23} can also be represented by the Fourier multiplier operator:
\begin{equation*}
\mathcal{A} u(x)=\frac{1}{(2\pi)^d}\int_{\mathbb{R}^3}\left(-|\xi_1|^{\alpha_1}-|\xi_2|^{\alpha_2}-|\xi_3|^{\alpha_3}\right)\mathcal{F}[u](\xi)\mathrm{e}^{ix\cdot\xi}\mathrm{d}\xi,\quad \xi:=(\xi_1,\xi_2,\xi_3),
\end{equation*}
where $\mathcal{F}[u]$ is the Fourier transform of $u$. Since the Fourier multiplier of $\mathcal{A}$ is not radially symmetric, the operator $\mathcal{A}$ is not radially symmetric.

These two perspectives highlight the need for mathematical foundations for anisotropic diffusion models that differ from isotropic diffusion models.
To address this challenge, it is crucial to investigate stochastic processes and corresponding PDEs that are more general than those of SBM.
Additionally, there is a necessity to expand the mathematical tools employed in the study of isotropic diffusion models to accommodate the complexities of anisotropic diffusion models.

In this paper, we establish an $L_{q}(L_{p})$-regularity theory for parabolic equations with non-local operators having a time-measurable anisotropic symbol. Specifically, we consider the equation
\begin{align}
\label{23.02.22.15.28}
    \partial_{t}u(t,\vec{x}) = \mathcal{L}^{\vec{a},\vec{b}}(t)u(t,\vec{x}) +f(t,\vec{x}), \quad \vec{x}=(x_1,\cdots,x_{\ell})\in \mathbb{R}^{d},
\end{align}
where $d=\sum_{i=1}^{\ell} d_{i}$. Here $\mathcal{L}^{\vec{a},\vec{b}}(t)$ is a time-measurable anisotropic non-local operator defined by
\begin{equation}\label{eqn 02.15.17:45}
\begin{aligned}
&\mathcal{L}^{\vec{a},\vec{b}}(t)u(t,\vec{x})\\
&:=\sum_{j=1}^{\ell}\Big(b_j(t)\Delta_{x_j}u(t,\vec{x})\\
&\quad+\int_{\mathbb{R}^{d_j}}\left(u(t,\vec{x}+\vec{y})-u(t,\vec{x})-\vec{y}\cdot\nabla_{\vec{x}}u(t,\vec{x})1_{|\vec{y}|\leq1}\right)a_j(t,\vec{y})J_j(\mathrm{d}\vec{y})\Big),
\end{aligned}
\end{equation}
\begin{equation*}
    J_j(\mathrm{d}\vec{y}):=\mu_{j}(\mathrm{d}y_j)\epsilon_0^{d-d_j}(\mathrm{d}y_1,\cdots,\mathrm{d}y_{j-1},\mathrm{d}y_{j+1},\cdots,\mathrm{d}y_{\ell}),
\end{equation*}
where $\mu_j$ is a L\'evy measure of SBM, $\vec{y}=(y_1,y_2,\cdots,y_{\ell})\in\mathbb{R}^{d_1}\times\mathbb{R}^{d_2}\times\cdots\times\mathbb{R}^{d_{\ell}}=:\mathbb{R}^{\vec{d}}$, $a_j(t,\vec{y})$ and $b_j(t)$ are bounded measurable coefficients (see \eqref{23.03.06.12.24} for detail), and $\epsilon_{0}^{d-d_j}$ is the centered Dirac measure defined on $\mathbb{R}^{d-d_j}$.
We would like to emphasize that our results hold without any assumptions on the regularity or symmetry of coefficients $a_j(t,\vec{y})$ and $b_j(t)$.
Our main result is to prove the existence of a unique solution for \eqref{23.02.22.15.28} satisfying the following estimate:
$$
\int_{0}^{T} \left( \int_{\mathbb{R}^{d}} |\mathcal{L}^{\vec{a},\vec{b}}(t)u(t,x)|^{p} + |u(t,x)|^{p} \mathrm{d}x \right)^{q/p} \mathrm{d}t \leq C \int_{0}^{T} \left( \int_{\mathbb{R}^{d}} |f(t,x)|^{p} \mathrm{d}x  \right)^{q/p} \mathrm{d}t
$$
for $p,q\in(1,\infty)$, and $0<T<\infty$.

The operator $\mathcal{L}^{\vec{a},\vec{b}}(t)$ in \eqref{eqn 02.15.17:45} generalizes the singular fractional Laplacian in \eqref{23.02.21.17.23}.
Since $\mathcal{L}^{\vec{a},\vec{b}}(t)$ is not radial and its symbol is quite singular, the approach based on radial decay of symbol (see \textit{e.g.} \cite{kim2013parabolic,kim16,KKL15}) cannot be applied directly. 
Similarly, it seems nontrivial to apply \cite[Theorem 8.7]{P13} to treat $L_{q}(L_{p})$-theory for our equation. See Remark \ref{rmk 02.14.16:17} for more details.

To overcome these difficulties, we use both the Calder\'on-Zygmund theory and stochastic process theory. The probability theory allows us to solve enough \eqref{23.02.22.15.28} with coefficient $\vec{a}(t,\vec{y})=(1,\dots,1)\in\mathbb{R}^{\ell}$, and $\vec{b}(t)=\vec{b}\in\mathbb{R}^{\ell}$, where $\vec{b}$ is a constant vector. Then, we exploit the independence of stochastic processes to represent the fundamental solution of \eqref{23.02.22.15.28}, denoted by $p(t,x)$, as a product of transition densities $p_i(t,x_i)$ corresponding to $d_i$-dimensional SBMs with Bernstein functions $\phi_i$;
$$
p(t,x)=p_1(t,x_1)\times \cdots\times p_{\ell}(t,x_{\ell})=\prod_{i=1}^{\ell}p_i(t,x_i).
$$
To estimate $p_i$, we impose a lower scaling condition on $\phi_i$:
\begin{align}\label{23.02.22.16.03}
c_i \left(\frac{R}{r}\right)^{\delta_i}\leq\frac{\phi_i(R)}{\phi_i(r)}, \qquad 0<r<R<\infty \quad (c_i>0, \delta_i\in(0,1]).
\end{align}
This condition is commonly used in the study of SBM (see \textit{e.g.} \cite{CK10,kim2014global,mimica2016heat}).
Under the lower scaling condition \eqref{23.02.22.16.03}, we can obtain upper bounds for each transition density $p_i$ as well as its arbitrary order derivative (see Theorem \ref{pestimate}). Based on the upper bounds of $p_i$, by applying the Calder\'on-Zygmund theory, we can obtain a BMO estimates for the solution (see Theorem \ref{23.02.22.17.33}):
$$
\left\|\mathcal{L}^{\vec{a},\vec{b}} u\right\|_{BMO(\mathbb{R}^{d+1})}\leq C\|f\|_{L_{\infty}(\mathbb{R}^{d+1})}.
$$
Since the $L_2$-estimates of the solution can be obtained easily from the Fourier analysis (see Lemma \ref{22estimate}), we use the Marcinkiewicz interpolation theorem to derive $L_p$-estimates of the solution. Based on $L_p$-estimates, we apply the $L_p$-valued Calder\'on-Zygmund theory to obtain $L_q(L_p)$ estimates of the solution.

Let us discuss a selection of results from the relevant literature. 
The form
\begin{equation}
\label{23.04.25.11.15}
    \mathcal{L} u(x) = \int_{\mathbb{R}^{d}} \left( u(x+y) - u(x) -(\nabla u(x) \cdot y) \chi_{\alpha}(y) \right)  \frac{m^{(\alpha)}(x,y)}{|y|^{d+\alpha}} \mathrm{d}y, \quad \alpha\in(0,2)
\end{equation}
is one of the simplest anisotropic non-local operators.
To the best of the authors knowledge, for the $L_p$-theory of parabolic equations with \eqref{23.04.25.11.15} was first developed in \cite{mik1992} (see also references therein). 
Here,  $m^{(\alpha)}(x,y)$ is uniformly continuous in $x$, differentiable in $y$ upto order $\lfloor d/2 \rfloor +1$, and $\chi_{\alpha}(y)=\mathbf{1}_{\alpha=1}\mathbf{1}_{|y|\leq 1} + \mathbf{1}_{\alpha\in(1,2)}$.
The $L_p$-theory of parabolic equations with operators of the form \eqref{23.04.25.11.15} has been investigated from various angles.
One line of research involves relaxing the regularity assumptions for $m^{(\alpha)}(x,y)$, as in \cite{DJK23nonlocal,DY23nonlocal,DK12elliptic,MP14cauchy,zhang2013lp}.
In \cite{DJK23nonlocal,DY23nonlocal}, for example, $m^{(\alpha)}(x,y)$ is merely measurable in $y$.
Another direction generalizes the principal part $|y|^{-d-\alpha}\mathrm{d}y$ to isotropic L\'evy measures with $m^{(\alpha)}(x,y)=1$ as discussed in \cite{kim2019L, zhang2013p}.
We note that the behavior of jumping kernels therein is dominated by isotropic jumping kernels.
Several results exist for anisotropic non-local operators, such as
$$
\mathcal{L}_{1}=-(-\Delta_{x_{1}})^{\alpha_1/2} - (-\Delta_{x_{2}})^{\alpha_2/2}.
$$
In \cite{mikulevivcius2017p, mikulevivcius2019cauchy}, the $L_p$-theory for parabolic equations with scalable operators was established, where a specific example of scalable operators is $\mathcal{L}_1$ with $\alpha_1=\alpha_2$.
However, an explicit description of $\mathcal{L}_1$ is not available in the literature. 
Recently, the $L_{p}$-theory of homogeneous parabolic equations was addressed in \cite{CK2020}, which partially fills this gap.
It is worth noting that the aforementioned results cover only the anisotropic non-local operators with the same order, such as $\alpha_1=\alpha_2$.

In addition to the $L_p$-theory, various researchers have investigated similar results for integro-differential equations with analogous integro-differential operators.
For instance, in \cite{depablo2020}, the well-posedness and regularity theory of very weak solutions to anisotropic non-local parabolic equations with singular forcing that generalize $\mathcal{L}$ by replacing $\mathrm{d}y$ with $\rho^{d-1}\mathrm{d}\rho \mu(\mathrm{d}\theta)$, with non-symmetric spherical measure $\mu(\mathrm{d}\theta)$, was also developed. 
However, these results also cover only the anisotropic non-local operator with same order.
We can also find researches on the H\"older regularity of solutions to equations associated with anisotropic non-local operators.
In \cite{BC10}, the authors examined a system of stochastic differential equations
\begin{equation*}
\mathrm{d} Y_t^i=\sum_{j=1}^{d}A_{ij}(Y_{t-})\mathrm{d} Z_t^j,
\end{equation*}
where $Z^1_t, \dots, Z^d_t$ are independent one-dimensional $\alpha$-stable processes.
They studied H\"older regularity of harmonic functions with respect to $Y=(Y^1, \dots, Y^d)$ under certain assumptions on the matrix $A$.
For the case where $Z^i$'s are independent $\alpha_i$-stable processes with different $\alpha_i$, a similar result was obtained in \cite{Chaker20} (see \cite{Chaker19} for the uniqueness of weak solutions for diagonal matrix $A$).  
In \cite{CK20}, the authors investigated Poisson equations with singular anisotropic operators such as $\mathcal{A}$ in \eqref{23.02.21.17.23}.
They achieved H\"older regularity of elliptic equations with singular anisotropic operators by using the weak Harnack inequality in the general framework of bounded measurable coefficients.
It is important to note that Harnack inequality for $\mathcal{A}$ in \eqref{23.02.21.17.23} does not hold even if $\alpha_1=\alpha_2=\alpha_3$ (\textit{cf}. \cite[Section 3]{BC10}).
The results in \cite{CK20} were extended to parabolic equations in \cite{CKW19}. See also \cite{KRS21} for the regularity of semigroup $P^Y_t$ associated with $Y$.
The regularity theory for fully nonlinear integro-differential equations associated with anisotropic non-local operators of the form
$$
\mathcal{L}_2u(x)=\int_{\mathbb{R}^d}\frac{u(x+y)-u(x)-\nabla u(x)\cdot y\mathbf{1}_{|y|\leq1}}{|y_1|^{d+\alpha_1}+\cdots+|y_d|^{d+\alpha_d}}\mathrm{d}y
$$
was developed in \cite{CLU14nonlinear}.
In \cite{Leit20}, the author investigated the anisotropic version of Caffarelli-Silvestre's extension problem (see \cite{C2007extension}) deducing $\mathcal{L}_2$. Both studies in \cite{CLU14nonlinear,Leit20} exploited the geometric aspect of the jump measure, which differs from $\pi(\mathrm{d}y)$ in \eqref{23.02.21.17.23}.
In \cite{SL21}, the H\"older regularity theory with operators more general than $\mathcal{L}_2$ was also investigated.
We also address a very recent result \cite{Leit2023} which develops $L_p$-regularity theory for elliptic equations with type of $\mathcal{L}_2$.

In this article, we only address the zero-initial value problem, as the trace space of the solution $u$ has not yet been characterized.
However, further research is currently underway, including contributions from the first author of this paper.
These efforts aim to propose a trace theorem applicable to our equations.
In the next publication based on the trace theorem, we plan to demonstrate the weighted regularity result with initial data.

We finish the introduction with some notations. We use $``:="$ or $``=:"$ to denote a definition. The symbol $\mathbb{N}$ denotes the set of positive integers and $\mathbb{N}_0:=\mathbb{N}\cup\{0\}$. Also we use $\mathbb{Z}$ to denote the set of integers. For any $a\in \mathbb{R}$, we denote $\lfloor a \rfloor$ the greatest integer less than or equal to $a$. As usual $\mathbb{R}^d$ stands for the Euclidean space of points $x=(x^1,\dots,x^d)$. We set
$$
B_r(x):=\{y\in \mathbb{R}^{d} : |x-y|<r\}, \quad \mathbb{R}_+^{d+1} := \{(t,x)\in\mathbb{R}^{d+1} : t>0 \}.
$$
For $i=1,\ldots,d$,
multi-indices $\sigma=(\sigma_{1},\ldots,\sigma_{d})$,
 and functions $u(t,x)$ we set
$$
\partial_{x^{i}}u=\frac{\partial u}{\partial x^{i}}=D_{i}u,\quad D^{\sigma}u=D_{1}^{\sigma_1}\cdots D_{d}^{\sigma_d}u,\quad|\sigma|=\sigma_{1}+\cdots+\sigma_{d}.
$$
We also use the notation $D_{x}^{m}$ for arbitrary partial derivatives of
order $m$ with respect to $x$.
For an open set $\mathcal{O}$ in $\mathbb{R}^{d}$ or $\mathbb{R}^{d+1}$, $C_c^\infty(\mathcal{O})$ denotes the set of infinitely differentiable functions with compact support in $\mathcal{O}$. By 
$\mathcal{S}=\mathcal{S}(\mathbb{R}^d)$ we denote the  class of Schwartz functions on $\mathbb{R}^d$.
For $p\geq1$, by $L_{p}$ we denote the set
of complex-valued Lebesgue measurable functions $u$ on $\mathbb{R}^{d}$ satisfying
\[
\left\Vert u\right\Vert _{L_{p}}:=\left(\int_{\mathbb{R}^{d}}|u(x)|^{p} \mathrm{d}x\right)^{1/p}<\infty.
\]
Generally, for a given measure space $(X,\mathcal{M},\mu)$, $L_{p}(X,\mathcal{M},\mu;F)$
denotes the space of all $F$-valued $\mathcal{M}^{\mu}$-measurable functions
$u$ so that
\[
\left\Vert u\right\Vert _{L_{p}(X,\mathcal{M},\mu;F)}:=\left(\int_{X}\left\Vert u(x)\right\Vert _{F}^{p}\mu(\mathrm{d}x)\right)^{1/p}<\infty,
\]
where $\mathcal{M}^{\mu}$ denotes the completion of $\mathcal{M}$ with respect to the measure $\mu$.
We also denote by $L_{\infty}(X,\mathcal{M},\mu;F)$ the space of all $\mathcal{M}^{\mu}$-measurable functions $f : X \to F$ with the norm
$$
\|f\|_{L_{\infty}(X,\mathcal{M},\mu;F)}:=\inf\left\{r\geq0 : \mu(\{x\in X:\|f(x)\|_F\geq r\})=0\right\}<\infty.
$$
If there is no confusion for the given measure and $\sigma$-algebra, we usually omit the measure and the $\sigma$-algebra. 
For any given function $f:X \to \mathbb{R}$, we denote its inverse (if it exists) by $f^{-1}$. Also, for $\nu \in \mathbb{R}\setminus\{-1\}$ and nonnegative function $f$, we denote $f^{\nu}(x)=(f(x))^{\nu}$.
We denote $a\wedge b := \min\{a,b\}$ and $a\vee b:=\max\{a,b\}$. By $\mathcal{F}$ and $\mathcal{F}^{-1}$ we denote the $d$-dimensional Fourier transform and the inverse Fourier transform respectively, i.e.
$$
\mathcal{F}[f](\xi):=\int_{\mathbb{R}^d} \mathrm{e}^{-i\xi\cdot x} f(x)\mathrm{d}x, \quad \mathcal{F}^{-1}[f](\xi):=\frac{1}{(2\pi)^d}\int_{\mathbb{R}^d} \mathrm{e}^{i\xi\cdot x} f(x)\mathrm{d}x. 
$$
For any $a,b>0$, we write $a\simeq b$ if there is a constant $c>1$ independent of $a,b$ such that $c^{-1}a\leq b\leq ca$. For any complex number $z$, we denote $\Re[z]$ and $\Im[z]$ as the real and imaginary parts of $z$. Finally if we write $N=N(\dots)$, this means that the constant $N$ depends only on what are in the parentheses. The constant $N$ can differ from line to line.

\mysection{Settings and Main results} \label{Main result sec}
\subsection{Settings}
\label{23.04.19.21.02}
In this subsection, we review the basic definition and properties of Bernstein functions and introduce the concept of IASBMs.

$\bullet$ \textbf{Definition and Representation of Bernstein functions.}
A \emph{Bernstein function} is an infinitely differentiable function $\phi:(0,\infty)\to(0,\infty)$ that satisfies two conditions: $\phi(0+)=0$ and $(-1)^nD^n\phi(\lambda)\leq0$ for all $n\in\mathbb{N}$ and $\lambda\in(0,\infty)$. In other words, the function $\phi$ is non-negative, non-decreasing and concave.
It is well-known that every Bernstein function can be represented in the following form (see, for instance, \cite[Theorem 3.2]{schbern}):
\begin{align}
\label{23.03.08.12.52}
    \phi(\lambda)=b\lambda+\int_{0}^{\infty}\left(1-\mathrm{e}^{-\lambda t}\right)\,\mu(\mathrm{d}t),
\end{align}
where $b\geq0$ is the \emph{drift} of $\phi$ and $\mu$ is a non-negative measure on $\mathbb{R}_+$ satisfying the condition
\begin{align*}
\int_0^{\infty}\left(1\wedge t\right)\mu(\mathrm{d}t)<\infty.
\end{align*}
The measure $\mu$ is called the \emph{L\'evy measure} of $\phi$.
For a vector of Bernstein functions $\vec{\phi}:=(\phi_1,\cdots,\phi_{\ell})$, we say that $\vec{b}:=(b_1,\cdots,b_{\ell})\in\mathbb{R}_+^{\ell}$ is the \emph{drift} of $\vec{\phi}$ if each $b_i$ is the drift of $\phi_i$. Similarly, if we consider a vector of L\'evy measures $\vec{\mu}(\mathrm{d}\vec{t}):=\left(\mu_{1}(\mathrm{d}t_1),\cdots,\mu_{\ell}(\mathrm{d}t_{\ell})\right)$, where $\vec{t}=(t_1,\cdots,t_{\ell})\in\mathbb{R}_+^{\ell}$, we say that $\vec{\mu}$ is the \emph{L\'evy measure} of $\vec{\phi}$ if each $\mu_i$ is the L\'evy measure of $\phi_i$.

$\bullet$ \textbf{Relation between Bernstein functions and subordinators.}
Bernstein functions have a natural connection with subordinators, which are real-valued increasing L\'evy processes that start at $0$. Specifically, it is well-known (\textit{e.g.} \cite{sato1999,schbern}) that a function $\phi$ is a Bernstein function if and only if it is the Laplace exponent of a subordinator. That is, for a given Bernstein function $\phi$, there exists a unique subordinator $S=(S_t)_{t\geq0}$ on a probability space $(\Omega, \mathcal{F}, \mathbb{P})$ such that
$$\mathbb{E}[\mathrm{e}^{-\lambda S_t}]:=\int_{\Omega} \mathrm{e}^{-\lambda S_t(\omega)} \,\mathbb{P}(\mathrm{d}\omega)=\mathrm{e}^{-t\phi(\lambda)},\quad \forall (t,\lambda)\in[0,\infty)\times\mathbb{R}_+.
$$
This means that the probability distribution of the subordinator $S_t$ at any fixed time $t>0$ is completely determined by the Bernstein function $\phi$, and vice versa.

$\bullet$ \textbf{Subordinate Brownian motion(SBM).} Let $B=(B_t)_{t\geq0}$ be a $d$-dimensional Brownian motion that is independent of the subordinator $S$. The process $X=(B_{S_t})_{t\geq0}$ is a SBM, and its infinitesimal generator is defined by
$$
\phi(\Delta_{x})f(x)=\phi(\Delta)f(x):=\lim_{t \downarrow 0} \frac{ \mathbb{E}[f(x+X_t)] -f(x)}{t}.
$$
It is well-known that $X$ is an $\mathbb{R}^d$-valued radial symmetric L\'evy process, and its characteristic exponent is given by $\phi(|\cdot|^2)$, \emph{i.e.},
\begin{equation*}
\mathbb{E}[\mathrm{e}^{iX_t\cdot\xi}] = \mathrm{e}^{-t\phi(|\xi|^2)},\quad\forall(t,\xi)\in[0,\infty)\times\mathbb{R}^d.
\end{equation*}
According to \cite[Theorem 31.5]{sato1999}, $\phi(\Delta_x)$ has the following two representations:
\begin{equation}\label{fourier200408}
    \begin{aligned}
\phi(\Delta_x)f(x) &= b\Delta_x f + \int_{\mathbb{R}^d}\left(f(x+y)-f(x)-\nabla_xf(x)\cdot y\mathbf{1}_{|y|\leq1}\right)J(y)\mathrm{d}y,\\
&= \mathcal{F}^{-1}[-\phi(|\cdot|^2)\mathcal{F}[f]](x),
\end{aligned}
\end{equation}
where $J(y) := j(|y|)$ and $j:(0,\infty)\rightarrow(0,\infty)$ is given by
$$
j(r)=\int_{(0,\infty)} (4\pi t)^{-d/2} \mathrm{e}^{-r^2/(4t)} \mu(\mathrm{d}t).
$$
Here $b$ is the drift of $\phi$ and $\mu$ is the L\'evy measure of $\phi$. Also we call $J(y)=j(|y|)$ the jump kernel of $\phi(\Delta_{x})$.

$\bullet$ \textbf{Independent Array of Subordinate Brownian motion(IASBM).}

Throughout the rest of this paper, we assume that the spatial dimension $d$ satisfies the following equation:
$$
d=\vec{d}\cdot\vec{1}=\sum_{i=1}^{\ell}d_{i}, \quad \vec{1}:=(1,\cdots,1),\,\vec{d}:=(d_1,\cdots,d_{\ell})\in \mathbb{N}^{\ell},
$$
where $d_i$ represents the $i$-th subdimension and $\vec{1}$ is a vector of all ones.
For any vector $\vec{x} \in \mathbb{R}^{\vec{d}} := \mathbb{R}^{d_{1}} \times \cdots \times \mathbb{R}^{d_{\ell}}$, we use the notation
$$
\vec{x}=(x_{1},\dots,x_{\ell}), \quad x_{i}=(x^{1}_{i},\dots,x^{d_{i}}_{i})\in \mathbb{R}^{d_{i}} \quad (i=1,\dots,\ell).
$$
We say that an $\mathbb{R}^{\ell}$-valued stochastic process $\vec{S} = (S^1,\cdots,S^{\ell})$ is a \emph{vector of subordinators with Laplace exponent $\vec{\phi}=(\phi_1,\cdots,\phi_{\ell})$} if $S^1, \cdots, S^{\ell}$ are subordinators satisfying $\mathbb{E}[\mathrm{e}^{-\lambda{S^{i}_{t}}}]=\mathrm{e}^{-t\phi_i(\lambda)}$ for all $1 \leq i \leq \ell$.
Let $\vec{B} = (B^1,\cdots,B^{\ell}) \in \mathbb{R}^{\vec{d}}$ be a vector of independent Brownian motions. It is easy to check that $\vec{B}$ is a $d$-dimensional Brownian motion.
We say that an $\mathbb{R}^{\vec{d}}$-valued stochastic process $\vec{X} = (X^1,\cdots,X^{\ell}) := (B^1_{S^1},\cdots,B^{\ell}_{S^{\ell}}) \in \mathbb{R}^{\vec{d}}$ is an IASBM corresponding to $\vec{\phi}$ if $X_1,\cdots,X_{\ell}$ are independent SBMs.

Now we investigate fundamental properties of IASBMs, including characteristic exponents and infinitesimal generators. Let $\vec{X}$ be an IASBM corresponding to $\vec{\phi}$. We note that $X$ is an $\mathbb{R}^d$-valued L\'evy process, and its characteristic exponent is given by
\begin{equation}
\label{23.04.18.11.25}
    \mathbb{E}[\mathrm{e}^{i\vec{\xi}\cdot \vec{X}_t}]=\prod_{i=1}^{\ell}\exp\left(-t\phi_i(|\xi_i|^2)\right),\quad \vec{\xi}=(\xi_1,\cdots,\xi_{\ell})\in\mathbb{R}^{\vec{d}}.
\end{equation}
Since each component of $\vec{X}$ is independent, the infinitesimal generator of $\vec{X}$ can be expressed as
\begin{align*}
&\lim_{t\downarrow0}\frac{\mathbb{E}[f(\vec{x}+\vec{X}_t)]-f(\vec{x})}{t}=\sum_{i=1}^{\ell}\phi_{i}(\Delta_{x_i})f(x)=:(\vec{\phi}\cdot\Delta_{\vec{d}})f(x),
\end{align*}
where $\Delta_{\vec{d}}:=(\Delta_{x_1},\Delta_{x_2},\cdots,\Delta_{x_{\ell}})$, $\Delta_{x_i}$ is the Laplacian operator on $\mathbb{R}^{d_i}$ and
$$
\phi_i(\Delta_{x_i})f(\vec{x}):=\lim_{t\downarrow0}\frac{\mathbb{E}[f(x_1\cdots,x_{i-1},x_i+X_t^i,x_{i+1},\cdots,x_{\ell})]-f(\vec{x})}{t},\quad 1\leq i\leq \ell.
$$
Now, we show that $(\vec{\phi}\cdot\Delta_{\vec{d}})f$ can be represented as given in equation \eqref{fourier200408}.
To begin with, we define the Fourier transform of a function $f:\mathbb{R}^{\vec{d}}\to\mathbb{R}$ with respect to the $i$-th coordinate as $\mathcal{F}_{d_{i}}[f] := \int_{\mathbb{R}^{d_{i}}} f(x_{1},\dots,x_{i},\dots,x_{\ell}) \mathrm{e}^{-i \xi_{i}\cdot x_{i}}\mathrm{d}x_{i}$, for $1\leq i\leq \ell$, where $d_{i}$ denotes the dimension of the $i$-th coordinate. We also define the inverse Fourier transform $\mathcal{F}^{-1}_{d_{i}}$ in a similar way.
Using these notions, we can express the operator $\vec{\phi}\cdot\Delta_{\vec{d}}$ as
\begin{align*}
    (\vec{\phi}\cdot\Delta_{\vec{d}})f(\vec{x})&=\vec{b}\cdot\Delta_{\vec{d}}f(\vec{x})+\int_{\mathbb{R}^{d}}(f(\vec{x}+\vec{y})-f(\vec{x})-\nabla_{\vec{x}}f(\vec{x})\cdot \vec{y}\mathbf{1}_{|\vec{y}|\leq1})\vec{1}\cdot\vec{J}(\mathrm{d}\vec{y})\\
    &=\mathcal{F}^{-1}_{d}\left[-\sum_{i=1}^{\ell}\phi_{i}(|\xi_i|^2)\mathcal{F}_{d_i}[f]\right](\vec{x}).\nonumber
\end{align*}
where $\vec{b}$ is the drift of $\vec{\phi}$, $\vec{J}(\mathrm{d}\vec{y})$ is a vector of L\'evy measures defined by
\begin{equation}\label{23.03.04.19.19}
\begin{gathered}
    \vec{J}(\mathrm{d}\vec{y})=(J_1(\mathrm{d}\vec{y}),\cdots,J_{\ell}(\mathrm{d}\vec{y})),\\ J_i(\mathrm{d}\vec{y}):=J_i(y_i)\mathrm{d}y_{i}\epsilon_0^i(\mathrm{d}y_1,\cdots,\mathrm{d}y_{i-1},\mathrm{d}y_{i+1},\cdots,\mathrm{d}y_{\ell}),
\end{gathered}
\end{equation}
$J_i(y_i):=\int_{(0,\infty)}(4\pi t)^{-d_{i}/2}\mathrm{e}^{-|y_i|^2/(4t)}\mu_{i}(\mathrm{d}t)$ and $\mu_i$ is the L\'evy measure of $\phi_i$.

\begin{remark}
\label{rmk 04.05.15:28}
When $\ell=1$, we can directly see that
$$
\vec{\phi}\cdot\Delta_{\vec{d}}=\phi(\Delta_x).
$$
\end{remark}

\subsection{Main results}
In this subsection, we present the main result of this paper. The theorem is built on the following key assumption.
\begin{assumption}
\label{23.03.03.16.04}
Suppose that $\phi_1,\cdots,\phi_{\ell}$ are Bernstein functions satisfying that there exist constants $\delta_0\in (0,1]$ and $c_{0}>0$  such that
\begin{equation}\label{e:H}
\begin{gathered}
c_{0} \left(\frac{R}{r}\right)^{\delta_{0}}\leq\min\left(\frac{\phi_{1}(R)}{\phi_{1}(r)},\cdots,\frac{\phi_{\ell}(R)}{\phi_{\ell}(r)}\right), \qquad 0<r<R<\infty.
\end{gathered}
\end{equation}
\end{assumption}

\begin{remark}
(i) Assumption \ref{23.03.03.16.04} implies that every $\phi_{i}$ satisfies scaling condition \eqref{e:H}. If $\phi_{i}(r)=r^{\alpha_{i}}$ with $\alpha_{i}\in(0,1]$ ($i=1,\dots,\ell$), then $c_{0}=1$, and $\delta_{0}=\min\{\alpha_{1},\dots,\alpha_{\ell}\}$ fulfill  \eqref{e:H}. Thus Assumption \ref{23.03.03.16.04} covers vectors consisting of stable processes and Brownian motions. Also,  Assumption \ref{23.03.03.16.04} combined with the concavity of $\phi$ gives 
 \begin{equation}
\label{phiratio}
c_{0}\left(\frac{R}{r}\right)^{\delta_0}\leq \frac{\phi_{i}(R)}{\phi_{i}(r)} \leq \frac{R}{r} , \qquad 0<r<R<\infty.
\end{equation}

(ii) Such scaling conditions are closely related to behavior of SBM (or its transition density). For example, if
$$
c_{1}\left( \frac{R}{r} \right)^{\delta_{1}} \leq \frac{\phi(R)}{\phi(r)} \leq c_{2}\left( \frac{R}{r} \right)^{\delta_{2}} \quad \forall\, 0<r<R<\infty, \quad (c_{1},c_{2}>0,\quad \delta_{1},\delta_{2}\in(0,1)),
$$
then the corresponding heat kernel is comparable to $
(\phi^{-1}(t^{-1}))^{d/2} \wedge t \phi(|x|^{-2})|x|^{-d}$
(\textit{e.g.} \cite[Section 5]{kim2014global}). We also remark that scaling conditions for characteristic exponent like \eqref{e:H} or for jump kernels are widely used in the study of Markov processes. We refer to \cite{BKKL19, BGR14a, G14harnack,GKK20markov,kim2013parabolic, KW22, KM12, kim2014global, mimica2016heat}.
\end{remark}

Next, we introduce Sobolev spaces associated with the operator $\vec{\phi}\cdot\Delta_{\vec{d}}$ will serve as our solution spaces.
\begin{definition}\label{defn defining}
Let $1<p,q<\infty$, $\gamma\in \mathbb{R}$, and $0<T<\infty$. 
For a tempered distribution $u$, we define $(1-\vec{\phi}\cdot\Delta_{\vec{d}})^{\gamma/2}u$ as
\begin{align*}
    \mathcal{F}[(1-\vec{\phi}\cdot\Delta_{\vec{d}})^{-\gamma/2}u](\vec{\xi})&:=\left( 1-\mathcal{F}[\vec{\phi}\cdot\Delta_{\vec{d}}](\vec{\xi}) \right)^{\gamma/2}\mathcal{F}[u](\vec{\xi})\\
    &:= \left( 1+\sum_{i=1}^{\ell}\phi_{i}(|\xi_{i}|^{2}) \right)^{\gamma/2}\mathcal{F}[u] (\vec{\xi}).
\end{align*}

$(i)$ The space $H^{\vec{\phi},\gamma}_{p}=H^{\vec{\phi},\gamma}_{p}(\mathbb{R}^{d})$ is a closure of $\mathcal{S}(\mathbb{R}^d)$ under the norm
$$
\|u\|_{H^{\vec{\phi},\gamma}_{p}}:= \|(1-\vec{\phi}\cdot\Delta_{\vec{d}})^{\gamma/2}u\|_{L_{p}}<\infty.
$$

$(ii)$ The space $H^{\vec{\phi},\gamma}_{q,p}(T)$ is a set of measurable mappings $u:(0,T)\to H^{\vec{\phi},\gamma}_{p}$ satisfying
$$
\|u\|_{H^{\vec{\phi},\gamma+2}_{q,p}(T)}:= \left(\int_{0}^{T} \|u(t,\cdot)\|^{q}_{H^{\vec{\phi},\gamma+2}_{p}} \mathrm{d}t\right)^{1/q} <\infty.
$$
We also denote $L_{q,p}(T):=H^{\vec{\phi},0}_{q,p}(T)$.

$(iii)$ We denote $C^{\infty}_{p}([0,T]\times \mathbb{R}^{d})$ as a collection of functions $u(t,x)$ such that $ D^{m}_{x}u \in C([0,T];L_{p})$ for all $m\in \mathbb{N}_{0}$. We also denote $C^{1,\infty}_{p}([0,T]\times \mathbb{R}^{d})$ as a collection of functions $u(t,x)$ such that $u\in C_p^{\infty}([0,T]\times\mathbb{R}^d)$ and $\partial_tD^{m}_{x}u \in C([0,T];L_{p})$ for all $m\in \mathbb{N}_{0}$

$(iv)$ We say that $u\in \mathbb{H}^{\vec{\phi},\gamma+2}_{q,p}(T)$ if there exists a defining sequence of functions $u_{n} \in C_{p}^{1,\infty}([0,T]\times\mathbb{R}^d)$ such that $u_{n}$ converges to $u$ in $H^{\vec{\phi},\gamma+2}_{q,p}(T)$ and $\partial_{t}u_{n}$ is Cauchy in $H^{\vec{\phi},\gamma}_{q,p}(T)$. In this case, we define $\partial_{t}u$ as the limit of $\partial_{t}u_{n}$ in $H^{\vec{\phi},\gamma}_{q,p}(T)$. For $u\in \mathbb{H}^{\vec{\phi},\gamma+2}_{q,p}(T)$, we naturally define
$$
\|u\|_{\mathbb{H}^{\vec{\phi},\gamma+2}_{q,p}(T)} := \|\partial_{t}u\|_{H^{\vec{\phi},\gamma}_{q,p}(T)} + \|u\|_{H^{\vec{\phi},\gamma+2}_{q,p}(T)}.
$$

$(v)$ For $u\in \mathbb{H}^{\vec{\phi},\gamma+2}_{q,p}(T)$, we say that $u\in \mathbb{H}^{\vec{\phi},\gamma+2}_{q,p,0}(T)$ if there is a defining sequence $u_{n}$ so that
$$
u_{n}(0,x)=0 \quad \forall\, n\in \mathbb{N}.
$$
\end{definition}
It is worth noting that if $\vec{\phi}(\lambda)=(\lambda,\cdots,\lambda)$, then $H_p^{\vec{\phi},\gamma}$ corresponds to the classical Bessel potential space $H^{\gamma}_{p}$. 
\begin{remark} \label{Hvaluedconti}
Note that the embedding $H_p^{2n} \subseteq H_p^{\vec{\phi},2n}$ is continuous for any $n\in\mathbb{N}$, as shown in \cite[Remark 3]{mikulevivcius2017p}.
\end{remark}

\begin{lemma}\label{H_p^phi,gamma space}
Let $1<p<\infty$ and $\gamma\in\mathbb{R}$.

$(i)$ The space $H_p^{\vec{\phi},\gamma}$ is a Banach space.

$(ii)$ For any $\mu\in\mathbb{R}$, the map $(1-\vec{\phi}\cdot\Delta_{\vec{d}})^{\mu/2}$ is an isometry from $H^{\vec{\phi},\gamma}_{p}$ to $H^{\vec{\phi},\gamma-\mu}_{p}$.

$(iii)$ If $\mu>0$, then we have continuous embeddings $H_p^{\vec{\phi},\gamma+\mu}\subset H_p^{\vec{\phi},\gamma}$ in the sense that
\begin{equation*}
\|u\|_{H_p^{\vec{\phi},\gamma}}\leq C \|u\|_{H_p^{\vec{\phi},\gamma+\mu}},
\end{equation*}
where the constant $C$ is independent of $u$. 

$(iv)$ For any $u\in H^{\vec{\phi},\gamma+2}_{p}$, we have
\begin{equation*}
\left(\|u\|_{H^{\vec{\phi},\gamma}_p}+\|(\vec{\phi}\cdot\Delta_{\vec{d}})u\|_{H^{\vec{\phi},\gamma}_p}\right) \simeq \|u\|_{H_p^{\vec{\phi},\gamma+2}}.
\end{equation*}
\end{lemma}

\begin{proof}
The first and second assertions are direct consequences of the definition. Recall that the function $
\psi(\xi) := \sum_{i=1}^{\ell} \phi_{i}(|\xi_{i}|^{2})$
is a continuous negative definite function since each $\phi_{i}(|\xi_{i}|^{2})$ is continuous and negative definite (see \textit{e.g.} \cite{farkaspsi} for detail). Hence, the third assertion comes from \cite[Theorem 2.3.1]{farkaspsi}. Finally, the last assertion can be obtained by using the second assertion and  \cite[Theorem 2.2.7]{farkaspsi}.
\end{proof}

By following the proof of  \cite[Lemma 2.7]{kim2022nonlocal}, replacing time non-local operator therein with usual derivative, and using Remark \ref{Hvaluedconti}, and Lemma \ref{H_p^phi,gamma space}, we have the following properties of solution spaces.

\begin{lemma} \label{basicproperty}
Let $1<p,q<\infty$, $\gamma\in\mathbb{R}$, and $0<T<\infty$.

$(i)$ The spaces $H_{q,p}^{\vec{\phi},\gamma}(T)$ and $\mathbb{H}_{q,p}^{\vec{\phi},\gamma}(T)$ are Banach spaces.

$(ii)$ The space $\mathbb{H}_{q,p,0}^{\vec{\phi},\gamma+2}(T)$ is a closed subspace of $\mathbb{H}_{q,p}^{\vec{\phi},\gamma+2}(T)$.

$(iii)$ $C_c^\infty(\mathbb{R}^{d+1}_+)$ is dense in $\mathbb{H}_{q,p,0}^{\vec{\phi},\gamma+2}(T)$.

$(iv)$ For any $\gamma,\nu\in\mathbb{R}$, $(1-\vec{\phi}\cdot\Delta_{\vec{d}})^{\nu/2}:\mathbb{H}_{q,p}^{\vec{\phi},\gamma+2}(T)\to\mathbb{H}_{q,p}^{\vec{\phi},\gamma-\nu+2}(T)$ is an isometry, and for any  $u\in\mathbb{H}_{q,p}^{\vec{\phi},\gamma+2}(T)$
\begin{equation*} \label{spacetimederiv change}
(1-\vec{\phi}\cdot\Delta_{\vec{d}})^{\nu/2}\partial_t  u=\partial_t  (1-\vec{\phi}\cdot\Delta_{\vec{d}})^{\nu/2} u.
\end{equation*}
\end{lemma}

Here is the main result of this article.

\begin{theorem} \label{main theorem}
Let $1<p,q<\infty$, $\gamma \in \mathbb{R}$, and $0<T<\infty$.
Suppose that $\vec{\phi}=(\phi_1,\cdots,\phi_{\ell})$ is a vector of Bernstein functions satisfying Assumption \ref{23.03.03.16.04} with drift $\vec{b}_{0}=(b_{01},\dots,b_{0\ell})$ and vector of L\'evy measures $\vec{J}(\mathrm{d}\vec{y})$ defined in \eqref{23.03.04.19.19}.
For vectors of measurable functions
$$
\vec{a}(t,\vec{y}):=(a_1(t,y_1),\cdots,a_{\ell}(t,y_{\ell})),\quad \vec{b}(t):=(b_1(t),\cdots,b_{\ell}(t))
$$
satisfying
\begin{equation}
\label{23.03.06.12.24}
    a_i(t,y_{i}),\in[c_1,c_1^{-1}],\quad b_{i}(t) \in [c_{1}b_{0i},c^{-1}_{1}b_{0i}]  \quad \forall \,i=1,\cdots,\ell, 
\end{equation}
with a positive constant $c_{1}$, define the operator $\mathcal{L}^{\vec{a},\vec{b}}(t)$ as
$$
\mathcal{L}^{\vec{a},\vec{b}}(t)h(\vec{x}):=\vec{b}(t)\cdot\Delta_{\vec{d}}h(\vec{x})+\int_{\mathbb{R}^{d}}(h(\vec{x}+\vec{y})-h(\vec{x})-\nabla_{\vec{x}}h(\vec{x})\cdot \vec{y}\mathbf{1}_{|\vec{y}|\leq1})\vec{a}(t,\vec{y})\cdot\vec{J}(\mathrm{d}\vec{y}).
$$
Then for any $f\in H_{q,p}^{\vec{\phi},\gamma}(T)$, the equation
\begin{equation}\label{mainequation1}
\partial_t  u(t,\vec{x}) = \mathcal{L}^{\vec{a},\vec{b}}(t)u(t,\vec{x}) + f(t,\vec{x}),\quad t>0,\vec{x}\in\mathbb{R}^{\vec{d}}\,; \quad u(0,\vec{x})=0,\quad \vec{x}\in\mathbb{R}^{\vec{d}}
\end{equation}
admits a unique solution $u$ in the class $\mathbb{H}_{q,p,0}^{\vec{\phi},\gamma+2}(T)$, and   we have
\begin{equation*}
\|u\|_{\mathbb{H}_{q,p}^{\vec{\phi},\gamma+2}(T)}\leq C  \|f\|_{H_{q,p}^{\vec{\phi},\gamma}(T)},
\end{equation*}
where $C=C(\vec{b}_{0},d,\delta_0,c_0,c_1,p,q,\gamma,T)$.
Moreover,
\begin{equation}
   \label{mainestimate-111}
\|(\vec{\phi}\cdot\Delta_{\vec{d}}) u\|_{H_{q,p}^{\vec{\phi},\gamma}(T)}+\|\mathcal{L}^{\vec{a},\vec{b}} u\|_{H_{q,p}^{\vec{\phi},\gamma}(T)}\leq C_0 \|f\|_{H_{q,p}^{\vec{\phi},\gamma}(T)},
\end{equation}
where $C_0=C_0(\vec{b}_{0},d,\delta_0,c_0,c_1,p,q,\gamma)$. 
\end{theorem}

\begin{remark}
Note that \eqref{23.03.06.12.24} means that if $b_{0i}=0$, then we assume that $b_{i}(t) \equiv 0$.
\end{remark}

\begin{remark}
\label{23.03.07.14.34}

In Theorem \ref{main theorem}, we establish the existence, uniqueness and maximal regularity of the solution $u$ to the equation \eqref{mainequation1} in the function space $\mathbb{H}^{\vec{\phi},2}_{q,p,0}(T)$. In this remark, we show that a lower regularity of the solution can be also established through the use of the Marcinkiewicz multiplier theorem (\textit{e.g.} \cite[Theorem 6.2.4]{grafakos2014classical}). 

($i$) Let $u\in H^{\vec{\phi},2}_{p}$ and
$$
\phi_{1}(\Delta_{x_{1}})^{\alpha_{1}} \dots \phi_{\ell}(\Delta_{x_{\ell}})^{\alpha_{\ell}} (1-\vec{\phi}\cdot\Delta_{\vec{d}})^{-1}u(\vec{x}):=\mathcal{F}_d^{-1}[m\mathcal{F}_d[u]](\vec{x}),
$$
where $\alpha_{1},\cdots,\alpha_{\ell}$ are nonnegative constants satisfying $\alpha_{1}+\dots+\alpha_{\ell}\leq1$, and
\begin{align}\label{23.03.08.23.33}
m(\vec{\xi}) := \frac{(\phi_{1}(|\xi_{1}|^{2}))^{\alpha_{1}} \times \dots \times (\phi_{\ell}(|\xi_{\ell}|^{2}))^{\alpha_{\ell}}}{1+\phi_{1}(|\xi_{1}|^{2}) + \dots +\phi_{\ell}(|\xi_{\ell}|^{2})} := \frac{A(\vec{\xi})}{B(\vec{\xi})}:=A(\vec{\xi})H(\vec{\xi}).
\end{align}
Using the fact that $(\phi_{i})^{\alpha_{i}}$ is also a Bernstein function (see \textit{e.g.} \cite[Corollary 3.8 (iii)]{schbern}) and inequality
\begin{align}
	\label{eqn 07.19.14.35}
\lambda^{n}|\phi^{(n)}(\lambda)| \leq C(n) \phi(\lambda),\quad \forall n\in\mathbb{N},\lambda>0,
\end{align}
we can derive the following inequality for $1\leq k\leq d$, where $j_1,\cdots,j_k\in\{1,\cdots,d\}$:
\begin{align}
\label{23.03.08.14.35}
|D_{\xi^{j_1}}\cdots D_{\xi^{j_k}}A(\vec{\xi})|
&\leq C(\alpha_{j_1},\cdots,\alpha_{j_k},d,k) A(\vec{\xi})\prod_{i=1}^{k}|\xi^{j_i}|^{-1}.
\end{align}
Here, the inequality \eqref{eqn 07.19.14.35} is derived from \eqref{23.03.08.12.52} and the inequality $t^{n}\mathrm{e}^{-t}\leq N(n)(1-\mathrm{e}^{-t})$. Similarly, we have
\begin{align}
\label{23.03.08.14.36}
|D_{\xi^{j_1}}\cdots D_{\xi^{j_k}}H(\vec{\xi})|\leq C(d,k) H(\vec{\xi})\prod_{i=1}^{k}|\xi^{j_i}|^{-1}.
\end{align}
Then, applying the product rule of differentiation to \eqref{23.03.08.23.33}, and then using Young's inequality, we obtain the following inequality for any $1\leq k\leq\ell$:
\begin{equation}\label{eqn 04.10.15:05}
|D_{\xi^{j_1}}\cdots D_{\xi^{j_k}}m(\vec{\xi})| \leq C(\alpha_{j_1},\cdots,\alpha_{j_k},d,k) \prod_{i=1}^{k}|\xi^{j_i}|^{-1}.
\end{equation}
Thus, according to \cite[Corollary 6.2.5]{grafakos2014classical}, we have
\begin{align*}
&\|\phi_{1}(\Delta_{x_{1}})^{\alpha_{1}} \dots \phi_{\ell}(\Delta_{x_{\ell}})^{\alpha_{\ell}}u\|_{L_{p}}\\
&= \|(1-\vec{\phi}\cdot\Delta_{\vec{d}})\phi_{1}(\Delta_{x_{1}})^{\alpha_{1}} \dots \phi_{\ell}(\Delta_{x_{\ell}})^{\alpha_{\ell}} (1-\vec{\phi}\cdot\Delta_{\vec{d}})^{-1}u\|_{L_{p}} \nonumber
\\
& \leq C \|(1-\vec{\phi}\cdot\Delta_{\vec{d}})u\|_{L_{p}} = \|u\|_{H^{\vec{\phi},2}_{p}} <\infty.
\end{align*}
Also, if $\alpha_{i}=1$ for some $i=1,\dots,\ell$, then by replacing $B(\vec{\xi})$ by $\phi_{1}(|\xi_{1}|^{2}) +\dots +\phi_{\ell}(|\xi_{\ell}|^{2})$, and following the above argument, we can check that
\begin{align}\label{23.03.08.23.11}
\|\phi_{i}(\Delta_{x_{i}})u\|_{L_{p}} \leq C \|\vec{\phi}\cdot \Delta_{\vec{d}}u\|_{L_{p}} \quad \forall \, i=1,\dots,\ell.
\end{align}

($ii$) Assume that $\phi_{n}$ ($n\in \{1,\dots,\ell\}$) satisfies
\begin{equation}\label{eqn 04.10.14:47}
\sup_{r>1} \frac{\phi_{n}(r)}{\phi_{i}(r)} \leq C(i,n) \quad \forall\, i=1,\dots,\ell,
\end{equation}
and define
\begin{equation*}
m_{n}(\vec{\xi}) := \frac{\phi_{n}(|\vec{\xi}|^{2})}{1+\phi_{1}(|\xi_{1}|^{2}) +\dots + \phi_{\ell}(|\xi_{\ell}|^{2})}.
\end{equation*}
 Then like \eqref{23.03.08.14.35}, we see that $A_{n}(\vec{\xi}) := \phi_{n}(|\vec{\xi}|^{2})$
satisfies 
$$
|D_{\xi^{j_1}}\cdots D_{\xi^{j_k}}A_{n}(\vec{\xi})|
\leq C A_{n}(\vec{\xi})\prod_{i=1}^{k}|\xi^{j_i}|^{-1},
$$
where $j_1,\cdots,j_k\in\{1,\cdots,d\}$. Combining this with \eqref{23.03.08.14.36}, we see that $m_{n}$ satisfies
\begin{equation*}
|D_{\xi^{j_1}}\cdots D_{\xi^{j_k}}m_{n}(\vec{\xi})| \leq C(d,k) m_{n}(\vec{\xi}) \prod_{i=1}^{k}|\xi^{j_i}|^{-1}.
\end{equation*}
It is obvious that
\begin{equation}\label{eqn 04.10.15:04}
|D_{\xi^{j_1}}\cdots D_{\xi^{j_k}}m_{n}(\vec{\xi})| \leq C(d,k)  \prod_{i=1}^{k}|\xi^{j_i}|^{-1} \quad \forall \, |\vec{\xi}|\leq 2.
\end{equation}
Now, suppose that $|\vec{\xi}|\geq2$ and let $J$ be the subset of $\{1,\dots,\ell\}$ consisting of entries $j$ such that
$1\leq |\vec{\xi}|/|\xi_{j}| \leq 2$. Then by \eqref{phiratio} and \eqref{eqn 04.10.14:47}, we have
\begin{align*}
m_{n}(\vec{\xi}) &\leq \frac{4}{|J|}\sum_{j\in J} \frac{\phi_{n}(|\xi_{j}|^{2})}{1+\phi_{1}(|\xi_{1}|^{2}) + \dots + \phi_{\ell}(|\xi_{\ell}|^{2})}\leq \frac{4}{|J|}\sum_{j\in J} \frac{\phi_{n}(|\xi_{j}|^{2})}{1+ \phi_{j}(|\xi_{j}|^{2})} \leq C.
\end{align*}
Combining this with \eqref{eqn 04.10.15:04}, we check that $m_{n}$ also satisfies \eqref{eqn 04.10.15:05}. Therefore, like \eqref{23.03.08.23.11} we have
\begin{align*}
\|\phi_{n}(\Delta) u\|_{L_{p}}
&= \|(1-\vec{\phi}\cdot\Delta_{\vec{d}})\phi_{n}(\Delta)(1-\vec{\phi}\cdot\Delta_{\vec{d}})^{-1}u\|_{L_{p}} \nonumber
\\
& \leq C \|(1-\vec{\phi}\cdot\Delta_{\vec{d}})u\|_{L_{p}} = \|u\|_{H^{\vec{\phi},2}_{p}} <\infty.
\end{align*}
This shows that $\phi_{n}(\Delta)u\in L_{p}$. For example, if we let $\phi_{i}(r)=r^{\delta_{i}}$ and $\delta_{n}=\min\{\delta_{1},\dots,\delta_{\ell}\}$, then $\Delta^{\delta_{n}/2}u \in L_{p}$. Moreover, if $\phi_{i}=\phi$ for all $i=1,\dots,\ell$, then
\begin{align*}
\frac{\phi(|\xi|^{2})}{\phi(|\xi_{1}|^{2}) + \dots + \phi(|\xi_{\ell}|^{2})} 
&\leq \frac{4}{|J|}\sum_{j\in J} \frac{\phi(|\xi_{j}|^{2})}{ \phi(|\xi_{j}|^{2})} \leq C \quad \forall \, \xi \in \mathbb{R}^{d},
\end{align*}
where $J$ is taken as above. Using this and applying the similar argument,
we further check that $\phi(\Delta)u \in L_{p}$, and 
\begin{equation}\label{eqn 04.19.17:54}
\|\phi(\Delta)u \|_{L_{p}} \leq C \|(\vec{\phi}\cdot \Delta_{\vec{d}})u\|_{L_{p}}.
\end{equation}
\end{remark}
\begin{remark}
The inequality \eqref{23.03.08.23.11} not only provides regularity information for $u$ with respect to each coordinate, but also yields an equivalent characterization of the $H_{p}^{\vec{\phi},2}$ space as follows:
\begin{equation}\label{eqn 04.17.13:45}
\|(\vec{\phi}\cdot \Delta_{\vec{d}})u\|_{L_{p}}\simeq \sum_{i=1}^{\ell}\|\phi_i(\Delta_{x_i})u\|_{L_{p}}.
\end{equation}
Combining this characterization with Lemma \ref{H_p^phi,gamma space}-$(iv)$, we further obtain
$$
\|u\|_{H_{p}^{\vec{\phi},2}}\simeq\|u\|_{L_{p}}+\|(\vec{\phi}\cdot\Delta_{\vec{d}})u\|_{L_{p}}\simeq \|u\|_{L_{p}}+\sum_{i=1}^{\ell}\|\phi_i(\Delta_{x_i})u\|_{L_{p}}.
$$
This equivalence allows us to understand the $H_{p}^{\vec{\phi},2}$ space in terms of both individual coordinate regularities and simultaneous regularities across all coordinates. In particular, if $\phi_1(\lambda)=\cdots=\phi_{\ell}(\lambda)=\lambda^{\alpha}$, then $H_{p}^{\vec{\phi},2}$ becomes the classical Bessel potential space $H_{p}^{2\alpha}$.
\end{remark}

The proof of our main result, Theorem \ref{main theorem}, crucially relies on two statements: the probabilistic representation of the solution(Lemma \ref{probrep}) and the $L_q(L_p)$-boundedness of the solution(Theorem \ref{23.03.07.13.46}). The probabilistic representation of the solution can be easily obtained from \cite[Theorem 2.1.5]{choi_thesis}. We provide the proof of $L_q(L_p)$-boundedness of the solution in Section \ref{23.03.08.15.15}, which is essential to the proof of Theorem \ref{main theorem} in this article.
\begin{lemma}(\cite[Theorem 2.1.5]{choi_thesis}) \label{probrep}
Let $X$ be an additive process (see \textit{e.g.} \cite[Definition 1.6.]{sato1999}) with the characteristic exponent
$$
\mathbb{E}[\mathrm{e}^{i\xi\cdot (X_t-X_s)}]=\exp\left(\int_s^t\int_{\mathbb{R}^d}(\mathrm{e}^{i\xi\cdot y}-1-i\xi\cdot y1_{|y|\leq 1})\nu_r(\mathrm{d}y)\mathrm{d}r\right).
$$
Then for any $f\in C_p^{\infty}([0,T]\times\mathbb{R}^d)$,
\begin{equation*}
\begin{gathered}
u(t,x):=\int_{0}^{t} \mathbb{E}[f(s,x+X_t-X_s)] \mathrm{d}s\in C_p^{1,\infty}([0,T]\times\mathbb{R}^d)
\end{gathered}
\end{equation*}
is a unique classical solution of the equation 
$$
\partial_tu(t,x)=\mathcal{L}_{X}(t)u(t,x)+f(t,x),\quad \forall (t,x)\in(0,T)\times\mathbb{R}^{d},
$$
where
$$
\mathcal{L}_X(t)g(x):=\lim_{h\downarrow0}\frac{\mathbb{E}[g(x+X_{t+h}-X_t)]-g(x)}{h}.
$$
\end{lemma}
\begin{theorem}
\label{23.03.07.13.46}
Let $\vec{X}:=(X^1,\cdots,X^{\ell})$ be an IASBM corresponding to $\vec{\phi}$. Suppose that Assumption \ref{23.03.03.16.04} holds. Then for any $p,q\in(1,\infty)$ and $f\in C_{c}^{\infty}(\mathbb{R}^{d+1})$,
\begin{align} \label{qpestimate}
\|\mathcal{G} f\|_{L_q(\mathbb{R};L_p(\mathbb{R}^{d}))} \leq C(\vec{b}_0,d,c_{0},\delta_{0},p,q) \|f\|_{L_q(\mathbb{R};L_p(\mathbb{R}^{d}))},
\end{align}
where
\begin{equation*}
\mathcal{G}:f\mapsto\mathcal{G} f=\mathcal{G} f(t,\vec{x}):=\int_{-\infty}^t\mathbb{E}[(\vec{\phi}\cdot\Delta_{\vec{d}})f(s,\vec{x}+\vec{X}_{t-s})]\mathbf{1}_{s>0}\mathrm{d}s.
\end{equation*}
\end{theorem}
\begin{remark}\label{rmk 02.14.16:17}
By \eqref{23.04.18.11.25},
\begin{align*}
    \mathcal{F}_{d+1}[\mathcal{G} f](\tau,\vec{\xi})&=\frac{\phi_1(|\xi_1|^2)+\cdots+\phi_{\ell}(|\xi_{\ell}|^2)}{i\tau+\phi_1(|\xi_1|^2)+\cdots+\phi_{\ell}(|\xi_{\ell}|^2)}\mathcal{F}_{d+1}[f\mathbf{1}_{\mathbb{R}_+}](\tau,\vec{\xi})\\
    &=:m(\tau,\vec{\xi})\mathcal{F}_{d+1}[f\mathbf{1}_{\mathbb{R}_+}](\tau,\vec{\xi}).
\end{align*}
We would like to make a remark regarding the $L_q(L_p)$-boundedness of $\mathcal{G}$ when $p=q\neq 2$. For the case $\ell=1$, and $p=q$, we refer to \cite[Lemma 6.6]{kim2013parabolic}, which shows the $L_p$-boundedness of $\mathcal{G}$ based on the Mikhlin multiplier theorem.
It is worth noting that when $\ell>1$, then the result is not a simple consequence of results on Fourier multipliers, such as Mikhlin's and Marcinkiewicz's theorem.

To illustrate this point, let us consider the simple case where
$$
\phi_{1}(r) =r^{\delta_{1}},\quad \phi_{2}(r)=r^{\delta_{2}},\quad d_{1}=d_{2}=1,\quad 0<\delta_1,\delta_2<1.
$$
Then we claim that the $3$-dimensional Fourier multiplier
\begin{align*}
m(\tau,\xi_1,\xi_2) := \frac{|\xi_{1}|^{2\delta_1} + |\xi_{2}|^{2\delta_2}}{i\tau + |\xi_{1}|^{2\delta_1} + |\xi_{2}|^{2\delta_2}}
\end{align*}
does not satisfy assumptions of both Mikhlin's and Marcinkiewicz's multiplier theorem simultaneously.
For these, taking the derivative with respect to $\xi_1$,
\begin{align*}
\frac{\partial}{\partial{\xi_{1}}}m(\tau,\xi_1,\xi_2) =\frac{2\delta_1 |\xi_1|^{2\delta_1-2} \xi_{1} (2\tau^2(|\xi_1|^{2\delta_1}+|\xi_2|^{2\delta_2})+i\tau((|\xi_1|^{2\delta_1}+|\xi_2|^{2\delta_2})^2-\tau^2))}{|\tau|^{2} + (|\xi_{1}|^{2\delta_1}+|\xi_{2}|^{2\delta_2})^{2}}.
\end{align*}
Using this, we see that
\begin{align*}
\left|\Re\left[ \frac{\partial}{\partial{\xi_{1}}}m(\tau,\xi_1,\xi_2) \right] \right|  
&\geq\frac{4\delta_1|\tau|^2|\xi_1|^{2\delta_1-1}(|\xi_1|^{2\delta_1}+|\xi_2|^{2\delta_2})}{(|\tau|+|\xi_{1}|^{2\delta_1}+|\xi_{2}|^{2\delta_2})^{2}}.
\end{align*}
If
$$
1<R<|(\tau,\xi_1,\xi_2)|=\left(|\tau|^2+|\xi_1|^2+|\xi_2|^2\right)^{1/2}<2R,
$$
then
$$
\left|\Re\left[ \frac{\partial}{\partial{\xi_{1}}}m(\tau,\xi_1,\xi_2) \right] \right|\geq C(\delta_1) R^{-2}|\tau|^2|\xi_1|^{2\delta_1-1}(|\xi_1|^{2\delta_1}+|\xi_2|^{2\delta_2}).
$$
Hence
\begin{align*}
&R^{-\frac{3}{2}+1} \left(\int_{R<|(\tau,\xi_1,\xi_2)|<2R} \left| \Re\left[ \frac{\partial}{\partial{\xi_{1}}}m(\tau,\xi_1,\xi_2) \right] \right|^{2} \mathrm{d}\xi_1\mathrm{d}\xi_2 \mathrm{d}\tau\right)^{1/2}\\
&\geq CR^{-\frac{5}{2}}\left( \int_{\frac{5R}{4\sqrt{3}}<|\tau|,|\xi_1|,|\xi_2|<\frac{7R}{4\sqrt{3}}}|\tau|^4|\xi_1|^{4\delta_1-2}(|\xi_1|^{2\delta_1}+|\xi_2|^{2\delta_2})^2 \mathrm{d}\xi_1\mathrm{d}\xi_2 \mathrm{d}\tau\right)^{1/2}\\
&=CR^{2\delta_1-1}(R^{2\delta_1}+R^{2\delta_2}).
\end{align*}
Therefore, if $\delta_1>1/4$ or $\delta_1+\delta_2>1/2$, then the last term above goes to infinity as $R\to\infty$. Hence, $m(\tau,\xi_1,\xi_2)$ does not satisfy the condition for the Mikhlin multiplier theorem (see \textit{e.g.} \cite[Theorem 6.2.7]{grafakos2014classical}). Similarly, for $2^j\leq|\tau|,|\xi_2|<2^{j+1}$
\begin{align*}
    \int_{2^j\leq|\xi_1|<2^{j+1}} \left| \Re\left[ \frac{\partial}{\partial{\xi_{1}}}m(\tau,\xi_1,\xi_2) \right] \right| \mathrm{d}\xi_1\geq C2^{(2\delta_1-1)j}(2^{2\delta_1j}+2^{2\delta_2j}).
\end{align*}
This certainly implies that $m(\tau,\xi_1,\xi_2)$ does not satisfy the condition for the Marcikiewicz multiplier theorem (see \textit{e.g.} \cite[Theorem 6.2.4]{grafakos2014classical}).
\end{remark}

To prove Theorem \ref{main theorem}, the continuity of the operator $\mathcal{L}^{\vec{a},\vec{b}}(t)$ on $L_p$ is required.
This was established in \cite{BB2007,BBB2011}.
For the sake of completeness, we provide the details.
\begin{lemma}
\label{23.10.25.16.59}
    Let $p\in(1,\infty)$ and $\mathcal{L}^{\vec{a},\vec{b}}(t)$ be the operator in Theorem \ref{main theorem}.
    Then
    $$
    \|\mathcal{L}^{\vec{a},\vec{b}}(t)f\|_{L_p}\leq C\|(\vec{\phi}\cdot\Delta_{\vec{d}})f\|_{L_p},\quad \forall t\in[0,T],
    $$
    where $C=C(p,c_1,\vec{b}_{0})$.
\end{lemma}
\begin{proof}
    Let
    \begin{align*}
        \mu_{t}(\mathrm{d}\vec{x})&:=2\sum_{i=1}^{\ell}b_i(t)\epsilon_{0}^{i}(\mathrm{d}x_1,\cdots,\mathrm{d}x_{i-1},\mathrm{d}x_{i+1},\cdots,\mathrm{d}x_{\ell})\sum_{k=1}^{d_i}\epsilon_{e_i^k}(\mathrm{d}x_i),\\
        \mu(\mathrm{d}\vec{x})&:=2\sum_{i=1}^{\ell}b_{0i}\epsilon_{0}^{i}(\mathrm{d}x_1,\cdots,\mathrm{d}x_{i-1},\mathrm{d}x_{i+1},\cdots,\mathrm{d}x_{\ell})\sum_{k=1}^{d_i}\epsilon_{e_i^k}(\mathrm{d}x_i),
    \end{align*}
    where $\epsilon_0^i$ and $\epsilon_{e_i^k}$ are the Dirac measures on $\mathbb{R}^{d-d_i}$ centered at origin, and centered at $e_i^k$, respectively.
    Here $\{e_i^1,\cdots,e_i^{d_i}\}$ is the canonical basis of $\mathbb{R}^{d_i}$.
    Then it can be easily checked that $\mu$ and $\mu_t$ are finite nonnegative Borel measures on $\mathbb{S}^{d-1}$ and
    $$
    \frac{1}{2}\int_{\mathbb{S}^{d-1}}|\vec{\xi}\cdot\vec{\theta}|^2\mu_t(\mathrm{d}\vec{\theta})=\sum_{i=1}^{\ell}b_i(t)|\xi_i|^2,\quad \frac{1}{2}\int_{\mathbb{S}^{d-1}}|\vec{\xi}\cdot\vec{\theta}|^2\mu(\mathrm{d}\vec{\theta})=\sum_{i=1}^{\ell}b_{0i}|\xi_i|^2.
    $$
    By \cite[Theorem 1.1]{BBB2011}, for fixed $t$,
    \begin{align*}
        \|\mathcal{L}^{\vec{a},\vec{b}}(t)f\|_{L_p}=\left\|T_m(t)(\vec{\phi}\cdot\Delta_{\vec{d}})f\right\|_{L_p}\leq C\|(\vec{\phi}\cdot\Delta_{\vec{d}})f\|_{L_p},
    \end{align*}
    where $C=C(p,c_1,\vec{b}_0)$,
    $$
    T_{m}(t)f:=\mathcal{F}_d^{-1}[m(t,\cdot)\mathcal{F}_d[f]]
    $$
    and
    $$
    m(t,\vec{\xi}):=\frac{\int_{\mathbb{R}^d}(1-\cos(\vec{y}\cdot\vec{\xi}))\vec{a}(t,\vec{y})\cdot\vec{J}(\mathrm{d}\vec{y})+\frac{1}{2}\int_{\mathbb{S}^{d-1}}|\vec{\xi}\cdot\vec{\theta}|^2\mu_t(\mathrm{d}\vec{\theta})}{\int_{\mathbb{R}^d}(1-\cos(\vec{y}\cdot\vec{\xi}))\vec{1}\cdot\vec{J}(\mathrm{d}\vec{y})+\frac{1}{2}\int_{\mathbb{S}^{d-1}}|\vec{\xi}\cdot\vec{\theta}|^2\mu(\mathrm{d}\vec{\theta})}.
    $$
    The lemma is proved.
\end{proof}

\hfill

Now we prove Theorem \ref{main theorem}.
\begin{proof}[Proof of Theorem \ref{main theorem}]
Due to Remark \ref{Hvaluedconti} and Lemma \ref{basicproperty} $(iii)$, we only need to prove case $\gamma=0$.

\textbf{Uniqueness.} Let $u\in\mathbb{H}_{q,p,0}^{\vec{\phi},2}(T)$ be a solution of
\begin{equation*}
\begin{cases}
\partial_t u(t,\vec{x})=\mathcal{L}^{\vec{a},\vec{b}}(t)u(t,\vec{x}),\quad &(t,\vec{x})\in(0,T)\times\mathbb{R}^{\vec{d}},\\
u(0,\vec{x})=0,\quad & \vec{x}\in\mathbb{R}^{\vec{d}}.
\end{cases}
\end{equation*}
Then by Lemma \ref{basicproperty} $(iii)$, there exists a defining sequence $u_n\in C_c^{\infty}(\mathbb{R}^d_+)$ of $u$. Put $f_n:=\partial_tu_n-\mathcal{L}^{\vec{a},\vec{b}}u_n\in C_p^{\infty}([0,T]\times\mathbb{R}^d)$, then by Lemma \ref{H_p^phi,gamma space} $(iv)$ and Lemma \ref{23.10.25.16.59}, as $n\to\infty$,
\begin{align*}
    \|f_n\|_{L_{q,p}(T)}&=\|\partial_tu_n-\mathcal{L}^{\vec{a},\vec{b}}u_n\|_{L_{q,p}(T)}\\
    &\leq \|\partial_t(u_n-u)\|_{L_{q,p}(T)}+C\|(\vec{\phi}\cdot\Delta_{\vec{d}})(u_n-u)\|_{L_{q,p}(T)}\\
    &\leq C\|u_n-u\|_{\mathbb{H}_{q,p}^{\vec{\phi},2}(T)}\to0.
\end{align*}
We also have
$$
u_n(t,\vec{x})=\int_{0}^t\mathbb{E}[f_n(s,\vec{x}+\vec{Z}_t-\vec{Z}_s)]\mathrm{d}s
$$
from Lemma \ref{probrep}, where $\vec{Z}$ is an additive process with the characteristic exponent
$$
\mathbb{E}[\mathrm{e}^{i\vec{\xi}\cdot(\vec{Z}_t-\vec{Z}_s)}]=\exp\left(\int_s^t\int_{\mathbb{R}^d}(\mathrm{e}^{i\vec{\xi}\cdot\vec{y}}-1-i\vec{\xi}\cdot\vec{y}\mathbf{1}_{|\vec{y}|\leq1})\vec{a}(r,\vec{y})\cdot\vec{J}(\mathrm{d}\vec{y})\mathrm{d}r\right).
$$
Therefore, by Minkowski's inequality,
$$
\|u_n\|_{L_{q,p}(T)}\leq C(p,q,T)\|f_n\|_{L_{q,p}(T)}\to0
$$
as $n\to\infty$. This certainly implies that $u=0$.

\textbf{Existence and estimates.} Let $\{f_n\}_{n\in\mathbb{N}}\subseteq C_c^{\infty}(\mathbb{R}_{+}^{d})$ be an approximating sequence of $f$ in $L_{q,p}(T)$. With the help of Lemma \ref{probrep},
$$
u_n(t,\vec{x})=\int_{0}^t\mathbb{E}[f_n(s,\vec{x}+\vec{Z}_t-\vec{Z}_s)]\mathrm{d}s
$$
is a unique classical solution of 
\begin{equation*}
\begin{cases}
\partial_t u_n(t,\vec{x})=\mathcal{L}^{\vec{a},\vec{b}}(t)u_n(t,\vec{x})+f_n(t,\vec{x}),\quad &(t,\vec{x})\in(0,T)\times\mathbb{R}^{\vec{d}},\\
u_n(0,\vec{x})=0,\quad & \vec{x}\in\mathbb{R}^{\vec{d}}.
\end{cases}
\end{equation*}
Clearly, $u_n(0,\vec{x})=0$ and $
\|u_n\|_{L_{q,p}(T)}\leq C(p,q,T)\|f_n\|_{L_{q,p}(T)}$.
If we have
\begin{equation}
    \label{23.03.07.14.52}
    \|\mathcal{L}^{\vec{a},\vec{b}}u_n\|_{L_{q,p}(T)}+\|(\vec{\phi}\cdot\Delta_{\vec{d}})u_n\|_{L_{q,p}(T)}\leq C\|f_n\|_{L_{q,p}(T)},
\end{equation}
then Remark \ref{Hvaluedconti} and Lemma \ref{basicproperty}-($i$) give us the existence and estimates. We divide the proof of \eqref{23.03.07.14.52} into two cases.

\textit{Case 1.} $\vec{a}(t,\vec{y})=\vec{1}$, and $\vec{b}(t) = \vec{b}_{0}$.

By Theorem \ref{23.03.07.13.46},
$$
\|\mathcal{L}^{\vec{a},\vec{b}}u_n\|_{L_{q,p}(T)}+\|(\vec{\phi}\cdot\Delta_{\vec{d}})u_n\|_{L_{q,p}(T)}=2\|\mathcal{G} f_n\|_{L_{q,p}(T)}\leq C\|f_n\|_{L_{q,p}(T)}.
$$

\textit{Case 2.} General case.

This case follows from Case 1. Let $\vec{c}_1:=(c_1,\cdots,c_1)\in\mathbb{R}^{\ell}$ and
\begin{equation*}
\begin{gathered}
\nu_t^0(B):=\int_0^t\int_{\mathbb{R}^d}\mathbf{1}_{B}(\vec{y})\vec{a}(s,\vec{y})\cdot\vec{J}(\mathrm{d}\vec{y})\mathrm{d}s,\\
\nu_t^1(B):=\int_0^t\int_{\mathbb{R}^d}\mathbf{1}_{B}(\vec{y})(\vec{a}(s,\vec{y})-\vec{c}_1)\cdot\vec{J}(\mathrm{d}\vec{y})\mathrm{d}s,\\
    D_0(t):=\int_0^t\mathrm{diag}(\vec{b}(s))\mathrm{d}s\in\mathbb{R}^{d\times d},\quad D_1(t):=\int_0^t\mathrm{diag}(\vec{b}(s)-c_1\vec{b}_0)\mathrm{d}s\in\mathbb{R}^{d\times d}.
\end{gathered}
\end{equation*}
Then for a Borel measurable set $B$ and $s\leq t$,
$$
\nu_s^i(B)\leq \nu_t^i(B),\quad i=0,1
$$
and a Borel measurable set $B'$ with $B'\subseteq\{\vec{x}\in\mathbb{R}^{\vec{d}}:|\vec{x}|\geq\varepsilon\}$, $\varepsilon>0$,
$$
\lim_{s\to t}\nu_s^i(B')=\nu_t^i(B'),\quad i=0,1.
$$
Similarly, we can also check for $\vec{z}\in\mathbb{R}^{\vec{d}}$ and for $s\leq t$,
$$
\langle \vec{z}, D_i(s)\vec{z}\rangle\leq \langle \vec{z},D_i(t)\vec{z}\rangle,\quad \lim_{s\to t}\langle \vec{z},D_i(s)\vec{z}\rangle=\langle\vec{z},D_i(t)\vec{z}\rangle,\quad i=0,1.
$$
Therefore, $D_0,D_1,\nu^0$ and $\nu^1$ satisfy assumptions in \cite[Theorem 9.8]{sato1999} and hence, there exist additive processes $\vec{Y}$, $\vec{Z}$ with triplets $(0,D_1(t),\nu_t^1)_{t\geq0}$ and $(0,D_0(t),\nu_t^0)_{t\geq0}$, respectively due to the theorem.

By Lemma \ref{probrep},
\begin{equation}
\label{23.03.30.12.38}
    v_n(t,\vec{x},\omega)=\int_0^t\mathbb{E}_{\omega'}[f_n(s,\vec{x}-\vec{Y}_{s}(\omega)+\vec{X}_{t-s}(\omega'))]\mathrm{d}s\in C_p^{1,\infty}([0,T]\times\mathbb{R}^d)
\end{equation}
is a unique solution of
\begin{equation*}
\begin{cases}
\partial_tv_n(t,\vec{x},\omega)=c_1(\vec{\phi}\cdot\Delta_{\vec{d}})v_n(t,\vec{x},\omega)+f_n(t,\vec{x}-\vec{Y}_t(\omega)),\quad &(t,\vec{x})\in(0,T)\times\mathbb{R}^{\vec{d}},\\
u_n(0,\vec{x})=0,\quad & \vec{x}\in\mathbb{R}^{\vec{d}},
\end{cases}
\end{equation*}
where $\vec{X}$ is an IASBM corresponding to $c_1\vec{\phi}=(c_1\phi_1,\cdots,c_1\phi_{\ell})$ and $\mathbb{E}_{\omega'}[h(\omega')]:=\int_{\Omega'}h(\omega')\mathbb{P}(\mathrm{d}\omega')$.
Let
$$
u_n(t,\vec{x}):=\mathbb{E}_{\omega}[v_n(t,\vec{x}+\vec{Y}_t(\omega),\omega)].
$$
Clearly, $\{u_n\}_{n\in\mathbb{N}}\subseteq C_p^{1,\infty}([0,T]\times\mathbb{R}^d)$ and $u_n$ satisfies $u_n(0,\vec{x})=0$.
By \eqref{23.03.30.12.38} and Fubini's theorem,
\begin{equation}
\label{23.03.06.14.41}
\begin{aligned}
u_n(t,\vec{x})&=\int_{0}^t
\mathbb{E}_{\omega}\mathbb{E}_{\omega'}[f_n(s,x+(\vec{Y}_t-\vec{Y}_s)(\omega)+\vec{X}_{t-s}(\omega'))]\mathrm{d}s.
\end{aligned}
\end{equation}
Fubini's theorem and \eqref{23.03.06.14.41} yield that
\begin{equation*}
\begin{aligned}
\mathcal{F}_d[u_n(t,\cdot)](\vec{\xi})&=\int_0^t\mathbb{E}_{\omega}\mathbb{E}_{\omega'}[\mathrm{e}^{i\vec{\xi}\cdot(\vec{X}_{t-s}(\omega)+\vec{Y}_t(\omega')-\vec{Y}_s(\omega'))}]\mathcal{F}_d[f_n(s,\cdot)](\vec{\xi})\mathrm{d}s\\
&=\mathcal{F}_d\left[\int_{0}^t\mathbb{E}[f_n(s,\cdot+\vec{Z}_t-\vec{Z}_s)]ds\right](\vec{\xi}).
\end{aligned}
\end{equation*}
Hence, by Lemma \ref{probrep}, $u_n$ is a unique classical solution of 
\begin{equation*}
\begin{cases}
\partial_t u_n(t,\vec{x})=\mathcal{L}^{\vec{a},\vec{b}}(t)u_n(t,\vec{x})+f_n(t,\vec{x}),\quad &(t,\vec{x})\in(0,T)\times\mathbb{R}^{\vec{d}},\\
u_n(0,\vec{x})=0,\quad & \vec{x}\in\mathbb{R}^{\vec{d}}.
\end{cases}
\end{equation*}
With the help of Case 1,
\begin{align*}
     \|(\vec{\phi}\cdot\Delta_{\vec{d}})u_n\|_{L_{q,p}(T)}\leq \mathbb{E}_{\omega} \left[\|(\vec{\phi}\cdot\Delta_{\vec{d}})v_n(\omega)\|_{L_{q,p}(T)}\right]\leq C  \|f_n\|_{L_{q,p}(T)},
\end{align*}
where $C=C(\vec{b}_0,d,\delta_0,c_0,c_1,p,q)$ and \eqref{23.03.07.14.52} holds. 
By  Lemma \ref{23.10.25.16.59},
\begin{align*}
    \|\mathcal{L}^{\vec{a},\vec{b}}u_n\|_{L_{q,p}(T)}\leq c_1\|(\vec{\phi}\cdot\Delta_{\vec{d}})u_n\|_{L_{q,p}(T)}\leq C(\vec{b}_0,d,\delta_0,c_0,c_1,p,q)\|f_n\|_{L_{q,p}(T)}.
\end{align*}
The theorem is proved.
\end{proof}

\mysection{Proof of Theorem \ref{23.03.07.13.46}}
\label{23.03.08.15.15}
This section is dedicated to proving Theorem \ref{23.03.07.13.46}.
Our main objective is to obtain BMO-$L_{\infty}$ estimates (Theorem \ref{23.02.22.17.33}):
$$
\|\mathcal{G} f\|_{BMO(\mathbb{R}^{d+1})}\leq C\|f\|_{L_{\infty}(\mathbb{R}^{d+1})}
$$
which will play a crucial role in proving Theorem \ref{23.03.07.13.46}.
We note that having BMO-$L_{\infty}$ and $L_2$-$L_2$ estimates enables us to apply the Marcinkiewicz interpolation theorem and the Fefferman-Stein theorem to obtain $L_p$-$L_p$ estimates for $p\geq2$. By utilizing the duality arguments and the Banach space-valued Calder\'on-Zygmund theorem (see, \textit{e.g.}, \cite[Theorem 4.1]{krylov2001caideron}), we can establish Theorem \ref{23.03.07.13.46}.

As a crucial step towards proving both Theorem \ref{23.03.07.13.46} and \ref{23.02.22.17.33}, we begin by proving Theorem \ref{23.03.07.13.46} with $p=q=2$. This lemma plays a fundamental role in our argument, and its proof relies on Plancherel's theorem.
\begin{lemma} \label{22estimate}
There exists a positive constant $C=C(d)$ such that for any $f\in C_c^\infty(\mathbb{R}^{d+1})$
\begin{equation*}
\|\mathcal{G} f\|_{L_2(\mathbb{R}^{d+1})}\leq C \|f\|_{L_2(\mathbb{R}^{d+1})}.
\end{equation*}
\end{lemma}

\begin{proof}
Recall that 
$$
\mathcal{G} f(t,\vec{x})=\int_{-\infty}^t\mathbb{E}[(\vec{\phi}\cdot\Delta_{\vec{d}})f(s,\vec{x}+\vec{X}_{t-s})]\mathbf{1}_{s>0}\mathrm{d}s
$$
Due to Plancherel's theorem and Minkowski's inequality, we have
\begin{align*}
    &\|\mathcal{G} f\|_{L_2(\mathbb{R}^{d+1})}^2\\
    &=C(d)\int_{\mathbb{R}^{\vec{d}}}\int_{0}^{\infty}\left|\int_0^{t}\sum_{i=1}^{\ell}\phi_i(|\xi|^2)\exp\left(-s\sum_{i=1}^{\ell}\phi_i(|\xi|^2)\right)\mathcal{F}_d[f(t-s,\cdot)](\vec{\xi})\mathrm{d}s\right|^2\mathrm{d}t\mathrm{d}\vec{\xi}\\
    &\leq C(d)\int_{\mathbb{R}^{\vec{d}}}\int_{0}^{\infty}|\mathcal{F}_{d}[f(t,\cdot)](\vec{\xi})|^2 \mathrm{d}t\left(\sum_{i=1}^{\ell}\phi_i(|\xi|^2)\int_0^{\infty}\exp\left(-s\sum_{i=1}^{\ell}\phi_i(|\xi|^2)\right)\mathrm{d}s\right)^{2}\mathrm{d}\vec{\xi}\\
    &=C(d)\|\mathcal{F}_d[f]\|_{L_2((0,\infty)\times\mathbb{R}^d)}\leq C(d)\|f\|_{L_2(\mathbb{R}^{d+1})}.
\end{align*}
The lemma is proved.
\end{proof}
It is worth noting that Assumption \ref{23.03.03.16.04} is not necessary for the proof of Lemma \ref{22estimate}. However, in order to obtain BMO-$L_{\infty}$ estimates, we utilize the upper bounds of the heat kernel associated with the IASBM. This is where Assumption \ref{23.03.03.16.04} becomes necessary. In other words, the assumption ensures that the heat kernel has certain decay properties, which are crucial in the derivation of the desired estimates. Therefore, in the subsequent sections, we assume that Assumption \ref{23.03.03.16.04} holds.

\subsection{Estimations of IASBM's heat kernel}
In this subsection, we focus on estimating the upper bounds of the heat kernel for the IASBM $\vec{X}=(X^1,\cdots,X^{\ell})$ to derive the desired BMO-$L_{\infty}$ estimates.
Since each component process $X^i$ is independent, the heat kernel of $\vec{X}$, denoted by $p(t,\vec{x})$, can be expressed as a product of the heat kernels of component processes $X^1,\cdots,X^{\ell}$:
\begin{equation}
\label{23.03.09.15.47}
    p(t,\vec{x})=\prod_{i=1}^{\ell}p_i(t,x_i),\quad \forall (t,\vec{x})\in(0,\infty)\times\mathbb{R}^{\vec{d}},
\end{equation}
where $p_i$ is the heat kernel of $X^i$. Therefore, by obtaining upper bounds of the heat kernels for each component process, we can obtain upper bounds for the IASBM.

Before delving further into the main result of this subsection, we first investigate some fundamental properties of the heat kernel for the component processes. Assumption \ref{23.03.03.16.04} implies that
$$
\exp\left(-t\phi_i(|\xi_i|^2)\right)\in L_1(\mathbb{R}^{d_i})\cap  L_{\infty}(\mathbb{R}^{d_i}), \quad i=1,\cdots,\ell,
$$
which in turn yields that $X^i$ has the transition density given by
\begin{equation}\label{eqn 7.20.1}
p_{i}(t,x_{i}):=p_{d_{i}}(t,x_{i})=\frac{1}{(2\pi)^{d_{i}}} \int_{\mathbb{R}^{d_{i}}} \mathrm{e}^{i \xi_{i} \cdot x_{i}} \mathrm{e}^{-t\phi_{i}(|\xi_{i}|^2)} \mathrm{d}\xi_{i}.
\end{equation}
By \cite[Section 30.1]{sato1999}, we know that for any $t>0$ and $x_{i}\in \mathbb{R}^{d_{i}}$,
\begin{eqnarray}
\label{eqn 7.20.2}
p_{i}(t,x_{i})=\int_{(0,\infty)} (4\pi s)^{-d_{i}/2} \exp \left(-\frac{|x_{i}|^2}{4s}\right) \eta^{i}_{t} (\mathrm{d}s),
\end{eqnarray}
where $\eta^{i}_{t}(\mathrm{d}s)$ is the distribution of $S^{i}_{t}$. This implies that $X^{i}_t$ is rotationally invariant in $\mathbb{R}^{d_i}$. For convenience, we introduce the notation $p_{i}(t,r):=p_{i}(t,x_i)$ if $r=|x_{i}|$. Furthermore, as shown in \eqref{eqn 7.20.2}, we see that $r\mapsto p_{i}(t,r)$ is a decreasing function. In order to establish upper bounds for the heat kernels of each component process, we provide the following lemma.
\begin{lemma} \label{gammaversionj}
Let Assumption \ref{23.03.03.16.04} hold and $\nu>0$.

(i) Then for each $i=1,\cdots,\ell$, 
there exists a constant $C_{i}=C(c_{0}, \delta_0,\nu)$ such that
\begin{equation} \label{int phi}
\int_{\lambda^{-1}}^\infty r^{-1}(\phi_{i}(r^{-2}))^{\nu}\mathrm{d}r \leq C_{i} (\phi_{i}(\lambda^2))^{\nu}, \quad \forall\, \lambda>0.
\end{equation}

(ii) Let $k\in \mathbb{N}_{0}$. Then 
\begin{equation}\label{eqn 02.02.18:47}
\int_{\mathbb{R}^{d_{i}}} \left(t^{-\nu k}\left(\phi_{i}^{-1}(t^{-1})\right)^{d_{i}/2}\wedge t^{\nu-\nu k}\frac{(\phi_{i}(|x_{i}|^{-2}))^{\nu}}{|x_{i}|^{d_{i}}}\right) \mathrm{d}x_{i} \leq C_{i} t^{-\nu k},
\end{equation}
where $C_{i}=C_{i}(d_{i},c_{0},\delta_{0},k,\nu)$.
\end{lemma}

\begin{proof}
(i) By the change of variables and \eqref{phiratio},
\begin{equation*}
\begin{aligned}
\int_{\lambda^{-1}}^{\infty}r^{-1}(\phi_{i}(r^{-2}))^{\nu}\mathrm{d}r&=\int_{1}^{\infty}r^{-1}(\phi_{i}(\lambda^{2}r^{-2}))^{\nu}\frac{(\phi(\lambda^{2}))^{\nu}}{(\phi(\lambda^{2}))^{\nu}}\mathrm{d}r\leq C_{i}(\phi_{i}(\lambda^{2}))^{\nu}.
\end{aligned}
\end{equation*}

(ii) We can observe that
\begin{align*}
&\int_{\mathbb{R}^{d_{i}}} \left(t^{-\nu k}\left(\phi_{i}^{-1}(t^{-1})\right)^{d_{i}/2}\wedge t^{\nu -\nu k}\frac{(\phi_{i}(|x_{i}|^{-2}))^{\nu}}{|x_{i}|^{d_{i}}}\right) \mathrm{d}x_{i}
\\
&= C_{i} \Big( \int_{|x_{i}| \leq (\phi^{-1}_{i}(t^{-1}))^{-1/2}} t^{-\nu k}\left(\phi_{i}^{-1}(t^{-1})\right)^{d_{i}/2}  \mathrm{d}x_{i}
\\
&\qquad +\int_{|x_{i}| > (\phi^{-1}_{i}(t^{-1}))^{-1/2}}t^{\nu -\nu k}\frac{(\phi_{i}(|x_{i}|^{-2}))^{\nu}}{|x_{i}|^{d_{i}}} \mathrm{d}x_{i}\Big)\\
&\leq C_{i}t^{-\nu k} + C_{i} \int_{(\phi^{-1}_{i}(t^{-1}))^{-1/2}}^{\infty} t^{\nu -\nu k} (\phi_{i}(\rho^{-2}))^{\nu} \rho^{-1} \mathrm{d}\rho
\\
&\leq C_{i}t^{-\nu k} + C_{i}t^{\nu -\nu k} (\phi_{i}(\phi_{i}^{-1}(t^{-1})))^{\nu} \leq C_{i} t^{-\nu k},
\end{align*}
where for the third inequality, we used \eqref{int phi}. The lemma is proved.
\end{proof}

Now we present upper bounds of the heat kernel for each component processes.
 
\begin{theorem} \label{pestimate}
Let $i=1,\cdots\ell$ and Assumption \ref{23.03.03.16.04} hold. 

(i) For any $m,k\in \mathbb{N}_{0}$, and $\nu\in(0,1)$, we have
\begin{align}\label{eqn 02.09.16:49}
|\phi_{i}(\Delta_{x_{i}})^{\nu k}D^{m}_{x_{i}}p_{i}(t,x_{i})| \leq C_{i}  \left(t^{-\nu  k}\left(\phi_{i}^{-1}(t^{-1})\right)^{\frac{d_{i}+m}{2}}\wedge t^{\nu-\nu k}\frac{(\phi_{i}(|x_{i}|^{-2}))^{\nu}}{|x_{i}|^{d_{i}+m}}\right),
\end{align} 
where the constant $C_{i}$ depends only on $c_{0},\delta_{0},d_{i},m,k,\nu$. In particular, we have

\begin{equation} \label{pestimate ineq}
|\phi_{i}(\Delta_{x_{i}})^{k}D^{m}_{x_{i}}p_{i}(t,x_{i})| \leq C_{i}  \left(t^{-k}\left(\phi_{i}^{-1}(t^{-1})\right)^{\frac{d_{i}+m}{2}}\wedge t^{\frac{1}{2}-k}\frac{(\phi_{i}(|x_{i}|^{-2}))^{1/2}}{|x_{i}|^{d_{i}+m}}\right),
\end{equation}

(ii) For any $k=0,1,\dots$, we have
\begin{equation}\label{eqn 01.06.17:16}
\int_{\mathbb{R}^{d_{i}}} |\phi_{i}(\Delta_{x_{i}})^{k}p_{i}(t,x_{i})| \mathrm{d}x_{i} \leq C_{i} t^{-k},
\end{equation}
where the constant $C_{i}$ depends only on $c_{0},\delta_{0},d_{i},k$.
\end{theorem}
\begin{proof}
If we have \eqref{eqn 02.09.16:49}, then by plugging $\nu=1/2$ and $k=2k$ in \eqref{eqn 02.09.16:49}, we obtain \eqref{pestimate ineq}. 
Similarly we also obtain \eqref{eqn 01.06.17:16} due to \eqref{eqn 02.02.18:47}. Hence, we only focus on proving \eqref{eqn 02.09.16:49}. By following \cite[Lemma A.3]{kim2022nonlocal} with $\phi^{\nu k}$ in place of $\phi$ therein, we have the estimation 
$$
|\phi_{i}(\Delta_{x_{i}})^{\nu k}D^{m}_{x_{i}}p_{i}(t,x_{i})| \leq C_{i}  t^{-\nu  k}\left(\phi_{i}^{-1}(t^{-1})\right)^{\frac{d_{i}+m}{2}}.
$$
Therefore, to finish the proof, we may assume
$$
t^{- \nu k}(\phi^{-1}_{i}(t^{-1}))^{\frac{d_{i}+m}{2}} \geq t^{\nu -\nu k} \frac{(\phi_{i}(|x_{i}|^{-2}))^{\nu}}{|x_{i}|^{d+m}}
$$
(equivalently, $t \phi_{i}(|x_{i}|^{-2})\leq 1$) and prove
\begin{align}\label{eqn 02.02.14:36}
|\phi_{i}(\Delta_{x_{i}})^{\nu k} D^m_{x_{i}} p_{i}(t,\cdot)(x_{i})|
\leq C_{i} t^{\nu - \nu k} \frac{(\phi_{i}(|x_{i}|^{-2}))^{\nu}}{|x_{i}|^{d+m}}.
\end{align}
We divide the proof of \eqref{eqn 02.02.14:36} into two steps.

\textbf{Step 1.} $k=0$.

When $k=0$, then \eqref{eqn 02.09.16:49} is a direct consequence of \cite[Lemma A.2 (ii)]{kim2022nonlocal}, which shows
\begin{align}\label{eqn 02.02.18:20}
|D^{m}_{x_{i}}p_{i}(t,x_{i})| &\leq C_{i} \sum_{0\leq n \leq \lfloor\frac{m}{2}\rfloor} |x_{i}|^{m-2n}\left( (\phi_{i}^{-1}(t^{-1}))^{d_{i}/2+m-n} \wedge   \frac{t\phi_{i}(|x_{i}|^{-2})}{|x_{i}|^{d_{i}+2(m-n)}} \right)    \nonumber
\\
&\leq C_{i}  \sum_{0\leq n \leq \lfloor\frac{m}{2}\rfloor} |x_{i}|^{m-2n}  \frac{t\phi_{i}(|x_{i}|^{-2})}{|x_{i}|^{d_{i}+2(m-n)}} \leq C_{i} \frac{(t\phi_{i}(|x_{i}|^{-2}))^{\nu}}{|x_{i}|^{d_{i}+m}}.
\end{align} 

For the remaining part $k\geq 1$, we use the mathematical induction argument.

\textbf{Step 2.} $k\geq1$.

Now, suppose that \eqref{eqn 02.02.14:36} (and hence \eqref{eqn 02.09.16:49}) holds for $k=k_{0}-1\geq 0$.
\\
By \cite[Lemma A.1]{kim2022nonlocal}, we see that $\phi_{i}^{\nu}$ is also a Bernstein function with L\'evy measure $\bar{\mu}$ (recall \eqref{fourier200408}), and 
\begin{align}\label{eqn 02.02.13:19}
\phi_{i}^{\nu}(\Delta_{x_{i}})f(x_{i})  &= \mathcal{F}_{d_i}[-\phi_i(|\cdot|^2)^{\nu}\mathcal{F}_{d_i}[f]]\nonumber\\
&=\int_{\mathbb{R}^{d_{i}}} \left(  f(x_{i}+y_{i}) - f(x_{i}) -\nabla_{x_{i}}f(x_{i}) \cdot y_{i}\mathbf{1}_{|y_{i}|\leq r} \right) \bar{j}_{d_{i}}(|y_{i}|) \mathrm{d}y_{i} \quad \forall \, r>0,  
\end{align}
where 
\begin{align}\label{defofjgamma}
\bar{j}_{d_{i}}(r):= \int_{0}^{\infty} (4\pi t)^{-d_{i}/2} \mathrm{e}^{-r^{2}/4t} \bar{\mu}(\mathrm{d}t),
\end{align}
and it satisfies
\begin{align}\label{jgamma ineq}
\bar{j}_{d_{i}}(r) \leq c(d_{i}) r^{-d_{i}} (\phi_{i}(r^{-2}))^{\nu} \quad \forall \, r>0.
\end{align}
For simplicity, denote $\phi_{i}(\Delta_{x_{i}})^{\nu k}D^{m}_{x_{i}}p_{i}(t,x_{i}):=P^{k}_{i,m}(t,x_{i})$ ($k\in\mathbb{N}_{0},m\in \mathbb{N}_{0}$). Then using \eqref{eqn 02.02.13:19} with $r=|x_{i}|/2$,
\begin{equation}
\label{eqn 02.02.14:44} 
\begin{aligned}
&\left|P^{k_{0}}_{i,m}(t,x_{i})\right|  \\
&\leq |P^{k_{0}-1}_{i,m}(t,x_{i})|\int_{|y_{i}|>\frac{|x_{i}|}{2}} \bar{j}_{d_{i}} (|y_{i}|)\mathrm{d}y_{i}
+ \left| \int_{|y_{i}|>\frac{|x_{i}|}{2}} P^{k_{0}-1}_{i,m}(t,x_{i}+y_{i})\bar{j}_{d_{i}} (|y_{i}|)\mathrm{d}y_{i}\right|  
\\
& \quad+ \int_{\frac{|x_{i}|}{2}>|y_{i}|} \int_0^1 \left| P^{k_{0}-1}_{i,m+1}(t,x_{i}+sy_{i})- P^{k_{0}-1}_{i,m+1}(t,x_{i})\right| |y_{i}|\bar{j}_{d_{i}}(|y_{i}|) \mathrm{d}s\mathrm{d}y_{i} 
\\
&=: |P^{k_{0}-1}_{i,m}(t,x_{i})|\times I + II_{k_{0}} + III_{k_{0}}. 
\end{aligned}
\end{equation}
By \eqref{jgamma ineq}, \eqref{int phi}, and \eqref{phiratio}, we have
\begin{align}\label{eqn 02.02.14:52}
I \leq C_{i} \int_{\frac{|x_{i}|}{2}}^{\infty} r^{-1}(\phi_{i}(r^{-2}))^{\nu} \mathrm{d}r \leq C_{i} (\phi_{i}(4|x_{i}|^{-2}))^{\nu} \leq C_{i} (\phi(|x_{i}|^{-2}))^{\nu}.
\end{align}
This together with  \eqref{eqn 02.02.18:20} yields (recall we assume $t\phi_{i}(|x_{i}|^{-2}) \leq 1$ and the assertion holds for $k=k_{0}-1$)
\begin{align*}
|P^{k_{0}-1}_{i,m}(t,x_{i})|\times I \leq C_{i} t^{\nu-\nu (k_{0}-1)} \frac{(\phi_{i}(|x_{i}|^{-2}))^{2\nu}}{|x_{i}|^{d_{i}+m}} \leq C_{i} t^{\nu- \nu k_{0}} \frac{(\phi_{i}(|x_{i}|^{-2}))^{\nu}}{|x_{i}|^{d_{i}+m}}.
\end{align*}
For $III_{k_{0}}$, by the fundamental theorem of calculus and \eqref{eqn 02.02.18:20},
\begin{align}\label{eqn 02.02.14:57}
III_{k_{0}} &\leq C(d_{i})\int_{\frac{|x_{i}|}{2}>|y_{i}|} \int_0^1 \int_0^1 \left|P^{k_{0}-1}_{i,m+2}(t,x_{i}+usy_{i})\right| |y_{i}|^2 \bar{j}_{d_{i}}(|y_{i}|) \mathrm{d}u\mathrm{d}s\mathrm{d}y_{i}  \nonumber
\\
&\leq C_{i} t^{\nu-\nu(k_{0}-1)}\int_{\frac{|x_{i}|}{2}>|y_{i}|} \int_0^1\int_0^1 \frac{(\phi(|x_{i}+usy_{i}|^{-2}))^{\nu}}{|x_{i}+usy_{i}|^{d_{i}+m+2}} |y_{i}|^2 \bar{j}_{d_{i}}(|y_{i}|) \mathrm{d}u\mathrm{d}s\mathrm{d}y_{i}  \nonumber
\\
&\leq C_{i} \frac {t^{-\nu(k_{0}-1)}(t\phi_{i}(|x_{i}|^{-2}))^{\nu}}{|x_{i}|^{d_{i}+m+2}}\int_{|x_{i}|/2>|y_{i}|} |y_{i}|^2 \bar{j}_{d_{i}}(|y_{i}|) \mathrm{d}y_{i}. 
\end{align}
For the last inequality above, we used $|x_{i}+usy_{i}|\geq |x_{i}|/2$, and \eqref{phiratio}. By \eqref{jgamma ineq}, and \eqref{phiratio} with $r=|x_{i}|^{-2}$ and $R=\rho^{-2}$, 
\begin{align}\label{eqn 02.02.15:04}
\int_{\frac{|x_{i}|}{2}>|y_{i}|} |y_{i}|^2 \bar{j}_{d_{i}}(|y_{i}|)\mathrm{d}y_{i} &\leq C_{i} \int_0^{|x_{i}|} \rho (\phi_{i}(\rho^{-2}))^{\nu} \mathrm{d}\rho \nonumber\\
&\leq  C_{i} |x_{i}|^2 (\phi_{i}(|x_{i}|^{-2}))^{\nu}.
\end{align}
Therefore, it follows that 
$$
III_{k_{0}}\leq  C_{i} \frac{t^{-\nu(k_{0}-1)}(t\phi_i(|x_{i}|^{-2}))^{\nu}}{|x_{i}|^{d_{i}+m+2}} |x_{i}|^{2}(\phi_i(|x_{i}|^{-2}))^{\nu}  \leq C_{i}  t^{\nu-\nu k_{0}} \frac{(\phi_{i}(|x_{i}|^{-2}))^{\nu}}{|x_i|^{d_i+m}}.
$$
Now we estimate $II_{k_{0}}$. By using the integration by parts $m$-times, we have
\begin{align*}
II_{k_{0}} &\leq \sum_{n=0}^{m-1} \int_{|y_{i}|=\frac{|x_{i}|}{2}} \left| \left(\frac{\mathrm{d}^{n}}{\mathrm{d}\rho^{n}} \bar{j}_{d_{i}}\right)(|y_{i}|) P^{k_{0}-1}_{i,m-1-n}(t,x_{i}+y_{i}) \right| S_{i}(\mathrm{d}y_i)\\
&\quad\quad+\int_{|y_{i}|>\frac{|x_{i}|}{2}} \left| \left(\frac{\mathrm{d}^{m}}{\mathrm{d}\rho^{m}}\bar{j}_{d_{i}}\right)(|y_{i}|)P^{k_{0}-1}_{i,0}(t,x_{i}+y_{i}) \right| \mathrm{d}y_{i}=:II'_{k_{0}} + II''_{k_{0}},
\end{align*}
where $S_{i}(\mathrm{d}y_i)$ is the surface measure defined on $\partial B_{|x_{i}|/2}$. Differentiating $\bar{j}_{d_{i}}$ via \eqref{defofjgamma}, and then using  \eqref{jgamma ineq} for $k\in\mathbb{N}_{0}$ we get
\begin{align}
\label{23.03.09.14.48}
\left| \frac{\mathrm{d}^{n}}{\mathrm{d}\rho^{n}} \bar{j}_{d_{i}}(\rho)  \right| 
\leq C_{i} \sum_{n-2l\geq 0,l\in\mathbb{N}_{0}} \rho^{n-2l}|\bar{j}_{d_{i}+2(n-l)}(\rho)|
\leq C_{i} \rho^{-d_{i}-n}(\phi_{i}(\rho^{-2}))^{\nu}.
\end{align}
Combining \eqref{23.03.09.14.48} and \eqref{eqn 02.02.18:20}, we have
\begin{align*}
II'_{k_{0}}
 &\leq C_{i}\sum_{k=0}^{m-1} \Big(  t^{\nu-\nu (k_{0}-1)} |x_{i}|^{d_{i}-1}|x|^{-d_{i}-k}  |x|^{-d_{i}-m+1+k}   (\phi_{i}(|x_{i}|^{-2}))^{2\nu} \Big)
\\
& \leq C_{i} t^{\nu-\nu k_{0}}  \frac{(\phi_{i}(|x_{i}|^{-2}))^{\nu}}{|x_{i}|^{d_{i}+m}}
    \end{align*}
when $t\phi_{i}(|x_{i}|^{-2}) \leq 1$. Also, using the fact that $\bar{j}_{d_{i}}$ is decreasing (recall \eqref{defofjgamma}), and \eqref{23.03.09.14.48}, we have
\begin{align*}
II''_{k_{0}} \leq C_{i} \frac{(\phi_{i}(|x_{i}|^{-2}))^{\nu}}{|x_{i}|^{d_{i}+m}} \int_{\mathbb{R}^{d_{i}}} |P^{k_{0}-1}_{i,0}(t,x_{i}+y_{i})| \mathrm{d}y_{i} \leq C_{i} t^{\nu-\nu k_{0}} \frac{(\phi(|x_{i}|^{-2}))^{\nu}}{|x_{i}|^{d_{i}+m}},
\end{align*}
where for the last inequality we used \eqref{eqn 02.02.18:47} for $k=k_{0}-1$. Therefore
$$
II_{k_{0}} \leq II'_{k_{0}}+II''_{k_{0}} \leq C_{i} t^{\nu -\nu k_{0}}  \frac{(\phi_{i}(|x_{i}|^{-2}))^{\nu}}{|x_{i}|^{d_{i}+m}}.
$$
Hence, we obtain \eqref{eqn 02.02.14:36} for $k=k_{0}$, and thus by the mathematical induction argument, we prove it for all $k\geq 1$. The theorem is proved.
\end{proof}

\subsection{BMO estimation of solution}
In this subsection, we establish the BMO-$L_{\infty}$ estimation of the operator $\mathcal{G}$:
\begin{equation}
\label{23.03.09.15.54}
    \|\mathcal{G} f\|_{BMO(\mathbb{R}^{d+1})}\leq C \|f\|_{L_{\infty}(\mathbb{R}^{d+1})},\quad \forall f\in L_2(\mathbb{R}^{d+1})\cap L_{\infty}(\mathbb{R}^{d+1}).
\end{equation}
We assume that Assumption \ref{23.03.03.16.04} holds and that the operator $\mathcal{G}$ can be expressed as a sum of $\ell$ terms, each of which we denote by $\mathcal{G}_i$:
\begin{align*}
    \mathcal{G} f(t,\vec{x})&=\mathbf{1}_{t>0}\sum_{i=1}^{\ell}\int_0^t\int_{\mathbb{R}^{\vec{d}}}\left(\prod_{j\neq i}p_j(t-s,x_j-y_j)\right)q_i(t-s,x_i-y)f(s,\vec{y})\mathrm{d}\vec{y}\mathrm{d}s\\
    &=:\sum_{i=1}^{\ell}\mathcal{G}_if(t,\vec{x}),
\end{align*}
where $q_i(t,x_i):=\phi(\Delta_{x_i})p_i$. To prove \eqref{23.03.09.15.54}, we need to show that
$$
\|\mathcal{G}_if\|_{BMO(\mathbb{R}^{d+1})}\leq C\|f\|_{L_{\infty}(\mathbb{R}^{d+1})},\quad i=1,\cdots,\ell.
$$
We will simplify the proofs by considering only the case $\ell=2$, since the same arguments can be applied for $\ell\geq3$. Due to the symmetry of the operators $\mathcal{G}_{1}$ and $\mathcal{G}_{2}$, we only need to consider the operator $\mathcal{G}_{1}$ in the rest of this section.

We begin by the notion of BMO spaces that will be used throughout this section. We define two increasing functions $\kappa_{1},\kappa_2:(0,\infty)\to (0,\infty)$ by
\begin{equation}\label{eqn 12.26.17:25}
\kappa_{i}(b):=\frac{1}{\sqrt{\phi^{-1}_{i}(b^{-1})}}, \quad b>0,\quad i=1,2.
\end{equation}
Here, $\phi_1$ and $\phi_2$ are Bernstein functions that satisfy Assumption \ref{23.03.03.16.04}. We also define $B^{i}_{r}(x_{i})$ ($i=1,2$) to be an open ball in $\mathbb{R}^{d_{i}}$ with radius $r>0$ and center $x_{i}$. For $(t,\vec{x})=(t,x_{1},x_{2}) \in \mathbb{R}^{d+1}$ and $b>0$, we denote
\begin{equation}\label{eqn 12.26.17:26}
Q_b(t, x_1, x_2):=(t-b,,t+b)\times B^{1}_{\kappa_{1}(b)}(x_{1})\times B^{2}_{\kappa_{2}(b)}(x_{2}),
\end{equation}
and $
Q_b=Q_b(0,0,0)$ and $B^{i}_{\kappa_{i}(b)}=B^{i}_{\kappa_{i}(b)}(0)$ for $i=1,2$.
For measurable subsets $Q\subset \mathbb{R}^{d+1}$ with finite measure and locally integrable functions $h$, we define the mean value of $h$ over $Q$ as
\begin{equation*}
h_Q:=\aint_{Q}h(s,\vec{y})\mathrm{d}\vec{y}\mathrm{d}s:=\aint_{Q}h(s,y_{1},y_{2})\mathrm{d}y_{1}\mathrm{d}y_{2}\mathrm{d}s:=\frac{1}{|Q|}\int_{Q}h(s,y_{1},y_{2})\mathrm{d}y_{1}\mathrm{d}y_{2}\mathrm{d}s,
\end{equation*}
where $|Q|$ is the Lebesgue measure of $Q$.
For a locally integrable function $h$ on $\mathbb{R}^{d+1}$, we define the BMO semi-norm of $h$ on $\mathbb{R}^{d+1}$ as
\begin{equation*}
\|h\|_{BMO(\mathbb{R}^{d+1})}:=\|h^{\#}\|_{L_{\infty}(\mathbb{R}^{d+1})}
\end{equation*}
where
\begin{align*}
h^{\#}(t,x):=\sup_{(t,x)\in Q_b(r,z)} \aint_{Q_b(r,z)}|h(s,y)-h_{Q_b(r,z)}|\mathrm{d}s\mathrm{d}y.
\end{align*}

\begin{remark}\label{rmk 02.16.17:37}
If $\ell>2$, we can naturally define $Q_{b}(t,\vec{x})$ as
$$
Q_{b}(t,\vec{x})=Q_{b}(t,x_{1},\dots,x_{\ell}) = (t-b,t+b) \times \prod_{i=1}^{\ell}B^{i}_{\kappa_{i}(b)}(x_{i}),
$$
where $\kappa_{i}(b)$ for other $i$ is defined as in \eqref{eqn 12.26.17:25}. Then, by following the rest of this section, we can establish the $L_{q}(L_{p})$-boundedness of $\mathcal{G}$ when $\ell>2$.
\end{remark}
Here is the main result of this subsection.
\begin{theorem}
\label{23.02.22.17.33}
For any $f\in L_2(\mathbb{R}^{d+1})\cap L_\infty(\mathbb{R}^{d+1})$,
\begin{align} \label{bmoestimate}
\|\mathcal{G}_1 f\|_{BMO(\mathbb{R}^{d+1})} \leq C(d,c_{0}, \delta_{0}) \|f\|_{L_\infty(\mathbb{R}^{d+1})}.
\end{align}
\end{theorem}
\begin{proof}
The proof is quite standard if we have the following Lemma;
\begin{lemma}
\label{outwholeestimate}
Let $f\in C_c^\infty(\mathbb{R}^{d+1})$ and $b>0$. Then,
\begin{align*}
\aint_{Q_b}\aint_{Q_b}|\mathcal{G}_{1} f (t,\vec{x})-\mathcal{G}_{1} f(s,\vec{y})|\mathrm{d}t \mathrm{d}\vec{x} \mathrm{d}s \mathrm{d}\vec{y} \leq C \|f\|_{L_\infty(\mathbb{R}^{d+1})},
\end{align*}
where $C$ depends only on $d, c_{0}$ and $\delta_{0}$.
\end{lemma}
Note that for any $(t_0,\vec{x}_0)\in\mathbb{R}^{d+1}$, by change of variables, we see that
\begin{eqnarray*}
\aint_{Q_b(t_0,\vec{x}_0)} |\mathcal{G}_1 f(t,\vec{x})-(\mathcal{G}_1 f)_{Q_b(t_0,\vec{x}_0)}| \mathrm{d}t \mathrm{d}\vec{x}
= \aint_{Q_b} |\mathcal{G}_1 \tilde{f}(t,\vec{x})-(\mathcal{G}_1 \tilde{f})_{Q_b}| \mathrm{d}t \mathrm{d}\vec{x},
\end{eqnarray*}
where $\tilde{f}(t,\vec{x}):=f(t+t_0, \vec{x}+\vec{x}_0)$. Using this and the invariance of  $L_{\infty}(\mathbb{R}^{d+1})$-norm under translation, we see that it  is enough to prove 
\begin{equation} \label{meanaverageineq}
\aint_{Q_b} |\mathcal{G}_1 f(t,\vec{x})-(\mathcal{G}_1 f)_{Q_b}| \mathrm{d}t \mathrm{d}\vec{x} \leq C \|f\|_{L_\infty(\mathbb{R}^{d+1})}, \quad b>0.
\end{equation}
Recall that we already have \eqref{meanaverageineq} due to Lemma \ref{outwholeestimate} if $f\in C_c^\infty (\mathbb{R}^{d+1})$. 

Now we consider the general case, \textit{i.e.}, $f\in L_2(\mathbb{R}^{d+1})\cap L_\infty(\mathbb{R}^{d+1})$.
We choose a sequence of functions $f_n\in C_c^\infty(\mathbb{R}^{d+1})$ such that $\mathcal{G}_1 f_n \to \mathcal{G}_1 f $ for almost, and $\|f_n\|_{L_\infty(\mathbb{R}^{d+1})}\leq \|f\|_{L_\infty(\mathbb{R}^{d+1})}$.
Then by Fatou's lemma,
\begin{align*}
\aint_{Q_b} |\mathcal{G}_1 f(t,\vec{x})-(\mathcal{G}_1 f)_{Q_b}| \mathrm{d}t \mathrm{d}\vec{x}& \leq \aint_{Q_b} \aint_{Q_b} |\mathcal{G}_1 f(t,\vec{x})-\mathcal{G}_1 f(s,\vec{y})| \mathrm{d}t \mathrm{d}\vec{x} \mathrm{d}s \mathrm{d}\vec{y}
\\
& \leq  \liminf_{n\to \infty} \aint_{Q_b} \aint_{Q_b} |\mathcal{G}_1 f_n(t,\vec{x})-\mathcal{G}_1 f_n(s,\vec{y})| \mathrm{d}t \mathrm{d}\vec{x} \mathrm{d}s \mathrm{d}\vec{y}
\\
& \leq C  \liminf_{n\to \infty} \|f_n\|_{L_\infty(\mathbb{R}^{d+1})}
\leq C \|f\|_{L_\infty(\mathbb{R}^{d+1})}.
\end{align*}
The theorem is proved.
\end{proof}

Now we prove Lemma \ref{outwholeestimate}.
\begin{proof}[Proof of Lemma \ref{outwholeestimate}]
Take functions $\eta=\eta(t) \in C^\infty(\mathbb{R})$ and $\zeta \in C_c^\infty(\mathbb{R}^{d})$ satisfying
\begin{itemize}
    \item $0\leq \eta\leq 1$, $\eta=1$ on $(-\infty, -8b/3)$ and $\eta(t)=0$ for $t\geq -7b/3$.
    \item $0\leq \zeta\leq 1$, $\zeta=1$ on $B^{1}_{7\kappa_{1}(b)/3}\times B^{2}_{7\kappa_{2}(b)/3}$ and $\zeta=0$ outside of $B^{1}_{8\kappa_{1}(b)/3}\times B^{2}_{8\kappa_{2}(b)/3}$.
\end{itemize}
Then for any $(t,\vec{x}),(s,\vec{y})\in\mathbb{R}\times\mathbb{R}^{\vec{d}}$,
\begin{equation*}
\begin{aligned}
|\mathcal{G}_{1}f(t,\vec{x})-\mathcal{G}_{1}f(s,\vec{y})|&\leq |\mathcal{G}_{1}f_1(t,\vec{x})-\mathcal{G}_{1}f_1(s,\vec{y})|+|\mathcal{G}_{1}f_2(t,\vec{x})-\mathcal{G}_{1}f_2(s,\vec{x})|\\
&\quad+ |\mathcal{G}_{1}f_3(s,\vec{x}) - \mathcal{G}_{1}f_3(s,\vec{y})| + |\mathcal{G}_{1}f_4(s,\vec{x}) - \mathcal{G}_{1}f_4(s,\vec{y})| 
\\
&=:I_1(t,s,\vec{x},\vec{y})+I_2(t,s,\vec{x},\vec{y})+I_3(t,s,\vec{x},\vec{y})+I_4(t,s,\vec{x},\vec{y}),
\end{aligned}
\end{equation*}
where
\begin{itemize}
    \item $f_1:=f(1-\eta)$; $f_1$ is supported in $(-3b,\infty)\times\mathbb{R}^{\vec{d}}$.
    \item $f_2:=f\eta$; $f_2$ is supported in $(-\infty,-2b)\times\mathbb{R}^{\vec{d}}$.
    \item $f_3:=f\eta(1-\zeta)$; $f_3$ is supported in $(-\infty,-2b)\times(B^{1}_{2\kappa_{1}(b)} \times B^{2}_{2\kappa_{2}(b)})^{c}$.
    \item $f_4:=f\eta\zeta$; $f_4$ is supported in
    $(-\infty,-2b)\times B^{1}_{2\kappa_{1}(b)} \times B^{2}_{2\kappa_{2}(b)}$.
\end{itemize}
Therefore, by proving
$$
\aint_{Q_b}\aint_{Q_b}(I_1+I_2+I_3+I_4)(t,s,\vec{x},\vec{y})\mathrm{d}t \mathrm{d}\vec{x} \mathrm{d}s \mathrm{d}\vec{y}\leq C\|f\|_{L_{\infty}(\mathbb{R}^{d+1})},
$$
we will achieve Lemma \ref{outwholeestimate}. 

\textbf{Step 1.} In Step 1, we prove
\begin{align*}
\aint_{Q_b}\aint_{Q_b}I_1(t,s,\vec{x},\vec{y}) \mathrm{d}\vec{x} \mathrm{d}t \vec{y} \mathrm{d}s&:=\aint_{Q_b}\aint_{Q_b}|\mathcal{G}_{1} f_1 (t,\vec{x})-\mathcal{G}_{1} f_1 (s,\vec{y})| \mathrm{d}\vec{x} \mathrm{d}t \vec{y} \mathrm{d}s\\
&\leq C \|f\|_{L_\infty(\mathbb{R}^{d+1})}.
\end{align*}
Recall that $f_1$ is supported in $(-3b,\infty)\times\mathbb{R}^{\vec{d}}$. Since
$$
\aint_{Q_b}\aint_{Q_b}|\mathcal{G}_{1} f_1 (t,\vec{x})-\mathcal{G}_{1} f_1 (s,\vec{y})| \mathrm{d}\vec{x} \mathrm{d}t \vec{y} \mathrm{d}s\leq 2\aint_{Q_b}|\mathcal{G}_1f_1(t,\vec{x})|\mathrm{d}\vec{x} \mathrm{d}t,
$$
it suffices to prove
\begin{align}
\label{23.03.10.13.13}
    \aint_{Q_b}|\mathcal{G}_1f_1(t,\vec{x})|\mathrm{d}\vec{x} \mathrm{d}t\leq C\|f\|_{L_{\infty}(\mathbb{R}^{d+1})}.
\end{align}
We divide the proof of \eqref{23.03.10.13.13} into two steps.

\textbf{Step 1-1.} The support of $f_1$ is contained in $(-3b,3b)\times B_{\kappa_1(3b)}\times B_{\kappa_2(3b)}$.

By the assumption and \eqref{phiratio}, $\|f_1\|_{L_2(\mathbb{R}^{d+1})}\leq C(d) |Q_{b}|^{1/2}\|f\|_{L_{\infty}(\mathbb{R}^{d+1})}$.
Thus,  by H\"older's inequality and Lemma \ref{22estimate},
\begin{align*}
\aint_{Q_b}|\mathcal{G}_1 f_1 (t,\vec{x})|\mathrm{d}\vec{x}\mathrm{d}t \leq & \left(\int_{Q_b}|\mathcal{G}_1 f_1 (t,\vec{x})|^2 \mathrm{d}\vec{x} \mathrm{d}t\right)^{1/2}|Q_b|^{-1/2}
\leq  C \|f\|_{L_{\infty}(\mathbb{R}^{d+1})}.
\end{align*}

\textbf{Step 1-2.} General case.

Take $\zeta_0=\zeta_0(t)\in C^{\infty}(\mathbb{R})$ such that $0\leq \zeta_0\leq 1$, $\zeta_0(t)=1$ for $t\leq 2b$, and $\zeta_0(t)=0$ for $t\geq 5b/2$.
Note that $\mathcal{G}_1 f=\mathcal{G}_1 (f\zeta_0)$ on $Q_b$ and $|f\zeta_0|\leq |f|$. This implies that to prove the lemma it is enough to assume $f(t,\vec{x})=0$ if $|t|\geq 3b$.

Set $f_{11}=\zeta f_{1}$ and $f_{12}=(1-\zeta)f_{1}$. Then $\mathcal{G}_1 f_{1} = \mathcal{G}_1 f_{11} + \mathcal{G}_1 f_{12}$.
Since $\mathcal{G}_1 f_{11}$ can be estimated by Step 1, to prove the lemma, we may further assume that $f(t,\vec{y})=0$ if $y\in B^{1}_{2\kappa_{1}(b)}\times B^{2}_{2\kappa_{2}(b)}$.  Therefore, for any $\vec{x}\in B^{1}_{\kappa_{1}(b)}\times B^{2}_{\kappa_{2}(b)}$,
\begin{align*}
&\int_{\mathbb{R}^{\vec{d}}} \left|\phi_{1}(\Delta_{x_{1}})p (t-s,\vec{x}-\vec{y}) f(s,\vec{y})\right| \mathrm{d}\vec{y}\\
&= \int_{(B^{1}_{2\kappa_{1}(b)}\times B^{2}_{2\kappa_{2}(b)})^{c}} |q_{1}(t-s,x_{1}-y_{1})p_{2}(t-s,x_{2}-y_{2}) f(s,\vec{y})| \mathrm{d}\vec{y} 
\\
&:= I_{1,1} + I_{1,2},
\end{align*}
where
\begin{gather*}
I_{1,1} = \int_{(B^{1}_{2\kappa_{1}(b)})^{c}\times \mathbb{R}^{d_2}} |q_{1}(t-s,x_{1}-y_{1})p_{2}(t-s,x_{2}-y_{2}) f(s,\vec{y})| \mathrm{d}\vec{y},
\\
I_{1,2} = \int_{(B^{1}_{2\kappa_{1}(b)}) \times (B^{2}_{2\kappa_{2}(b)})^{c}} |q_{1}(t-s,x_{1}-y_{1})p_{2}(t-s,x_{2}-y_{2}) f(s,\vec{y})| \mathrm{d}\vec{y}.
\end{gather*}
By \eqref{pestimate ineq} and \eqref{int phi},
\begin{align*}
I_{1,1} &\leq \|f\|_{L_\infty(\mathbb{R}^{d+1})} {\bf1}_{|s|\leq 3b} \int_{|y_{1}|\geq \kappa_{1}(b)} |q_{1} (t-s,y_{1})| \mathrm{d}y_{1}
\\
&\leq C\|f\|_{L_\infty(\mathbb{R}^{d+1})}{\bf1}_{|s|\leq 3b} (t-s)^{-1/2} \int_{(\phi_{1}^{-1}(b^{-1}))^{-1/2}}^{\infty} \frac{(\phi_{1}(\rho^{-2}))^{1/2}}{\rho^{d_{1}}}\rho^{d_{1}-1} \mathrm{d}\rho
\\
&\leq C\|f\|_{L_\infty(\mathbb{R}^{d+1})} {\bf1}_{|s|\leq 3b} (t-s)^{-1/2}  b^{-1/2}.
\end{align*}
Also, for $I_{1,2}$, again using \eqref{pestimate ineq}, \eqref{eqn 01.06.17:16} and \eqref{int phi}, we can check that 
\begin{align*}
I_{1,2} &\leq C\|f\|_{L_\infty(\mathbb{R}^{d+1})}{\bf1}_{|s|\leq 3b} (t-s)^{-1} \int_{(\phi_{2}^{-1}(b^{-1}))^{-1/2}}^{\infty} (t-s)^{1/2} \frac{(\phi_{2}(\rho^{-2}))^{1/2}}{\rho^{d_{2}}}\rho^{d_{2}-1} \mathrm{d}\rho
\\
&\leq C\|f\|_{L_\infty(\mathbb{R}^{d+1})}{\bf1}_{|s|\leq 3b} (t-s)^{-1/2} \int_{(\phi_{2}^{-1}(b^{-1}))^{-1/2}}^{\infty}  \frac{(\phi_{2}(\rho^{-2}))^{1/2}}{\rho^{d_{2}}}\rho^{d_{2}-1} \mathrm{d}\rho
\\
&\leq C\|f\|_{L_\infty(\mathbb{R}^{d+1})} {\bf1}_{|s|\leq 3b} (t-s)^{-1/2} b^{-1/2}.
\end{align*}
Note that if $|t| \leq b$ and $|s|\leq 3b$ then $|t-s|\leq 4b$.  Hence, it follows that for any $(t,\vec{x})\in Q_{b}$,  
\begin{align*}
|\mathcal{G}_{1} f(t,\vec{x})| & \leq \int_{-\infty}^{t} \left( I_{1,1}+I_{1,2}\right) \mathrm{d}s  \leq C\|f\|_{L_\infty(\mathbb{R}^{d+1})} b^{-1/2} \int_{|t-s|\leq 4b} |t-s|^{-1/2} \mathrm{d}s\\
&\leq C \|f\|_{L_{\infty}(\mathbb{R}^{d+1})}.
\end{align*}
This certainly implies the desired estimate \eqref{23.03.10.13.13}.

\textbf{Step 2.} In Step 2, we prove the estimation of $I_2$;
\begin{align*}
\aint_{Q_b}\aint_{Q_b}|\mathcal{G}_{1}f_2(t,\vec{x})-\mathcal{G}_{1}f_2(s,\vec{x})|\mathrm{d}\vec{x} \mathrm{d}t \vec{y} \mathrm{d}s\leq C \|f\|_{L_\infty(\mathbb{R}^{d+1})}.
\end{align*}
Recall that $f_2$ is supported in $(-\infty,-2b)\times\mathbb{R}^{\vec{d}}$. It suffices to prove 
\begin{align}
\label{23.03.10.01.27}
|\mathcal{G}_{1} f_2 (t_1,\vec{x})-\mathcal{G}_{1} f_2(t_2,\vec{x})|\leq C \|f\|_{L_\infty(\mathbb{R}^{d+1})},\quad \forall (t_1,\vec{x}), (t_2,\vec{x})\in Q_b.
\end{align}
Without loss of generality,  we assume $t_{1}>t_{2}$. Then, since $f_2(s,\vec{x})=0$ for $s\geq -2b$ and $t_1, t_2\geq -b$, it follows that
\begin{align*}
&|\mathcal{G}_{1} f_2(t_1,\vec{x})-\mathcal{G}_{1} f_2(t_2,\vec{x})|= \Big|\int_{-\infty}^{-2b}  \int_{\mathbb{R}^d} A(t_{1},t_{2},s,\vec{x},\vec{y})\,f(s,\vec{y}) \mathrm{d}\vec{y} \mathrm{d}s \Big|,
\end{align*}
where
\begin{equation*}
    A(t_{1},t_{2},s,\vec{x},\vec{y}) := q_{1}(t_{1}-s,x_{1}-y_{1})p_{2}(t_{1}-s,x_{2}-y_{2}) - q_{1}(t_{2}-s,x_{1}-y_{1})p_{2}(t_{2}-s,x_{2}-y_{2}).
\end{equation*}
By the fundamental theorem of calculus, we have
\begin{align*}
&\Big|\int_{-\infty}^{-2b}  \int_{\mathbb{R}^{\vec{d}}} A(t_{1},t_{2},s,\vec{x},\vec{y})f(s,\vec{y})  \mathrm{d}\vec{y} \mathrm{d}s \Big|\\
&\leq \Big|\int_{-\infty}^{-2b}  \int_{\mathbb{R}^{\vec{d}}}  \int_{t_2}^{t_1} \partial_{t}q_{1} (t-s,x_1-y_1) p_{2}(t_{1}-s,x_{2}-y_{2}) f(s,\vec{y})  \mathrm{d}t \mathrm{d}\vec{y} \mathrm{d}s \Big|
\\
& \quad +\Big|\int_{-\infty}^{-2b}  \int_{\mathbb{R}^{\vec{d}}}  \int_{t_2}^{t_1} q_{1} (t_{2}-s,x_1-y_1) q_{2}(t-s,x_{2}-y_{2}) f(s,\vec{y})  \mathrm{d}t \mathrm{d}\vec{y} \mathrm{d}s \Big|
\\
& =: I_{2,1}(t_{1},t_{2},\vec{x})+ I_{2,2}(t_{1},t_{2},\vec{x}).
\end{align*}
By using $\partial_{t}q_{1} = \partial_{t}\partial_{t}p_{1} = (\phi_{1}(\Delta_{x_{1}}))^{2}p_{1}$, and \eqref{eqn 01.06.17:16}, we can check that 
\begin{align*}
&I_{2,1}(t_{1},t_{2},\vec{x}) \\
&\leq \int_{-\infty}^{-2b} \left( \int_{\mathbb{R}^{d_{1}}} \int_{t_2}^{t_1} |\partial_{t}q_{1} (t-s,x_{1}-y_{1})f(s,\vec{y})| \mathrm{d}t \mathrm{d}y_{1} \int_{\mathbb{R}^{d_{2}}} p_{2}(t_{1}-s,x_{2}-y_{2})\mathrm{d}y_{2} \right) \mathrm{d}s 
\\
& \leq C \|f\|_{L_\infty(\mathbb{R}^{d+1})}  \int_{-\infty}^{-2b} 
 \int_{t_2}^{t_1} (t-s)^{-2}\mathrm{d}t \mathrm{d}s \leq C \int_{t_{2}}^{t_{1}} \|f\|_{L_{\infty}} b^{-1} \mathrm{d}s \leq C\|f\|_{L_{\infty}}.
\end{align*}
Similarly, we can check that
\begin{align*}
&I_{2,2}(t_{1},t_{2},\vec{x})\\
&\leq \int_{-\infty}^{-2b} \left( \int_{\mathbb{R}^{d_{2}}} \int_{t_2}^{t_1} |q_{2} (t-s,x_{2}-y_{2})f(s,\vec{y})| \mathrm{d}t \mathrm{d}y_{2} \int_{\mathbb{R}^{d_{1}}} |q_{1}(t_{2}-s,x_{1}-y_{1})|\mathrm{d}y_{1} \right) \mathrm{d}s 
\\
& \leq C \|f\|_{L_\infty(\mathbb{R}^{d+1})}  \int_{-\infty}^{-2b} \int_{t_2}^{t_1} (t-s)^{-1} (t_{2}-s)^{-1}  \mathrm{d}t\mathrm{d}s \\
&\leq C \|f\|_{L_\infty(\mathbb{R}^{d+1})}  \int_{-\infty}^{-2b} \int_{t_2}^{t_1} (t_{2}-s)^{-2}\mathrm{d}t  \mathrm{d}s \leq C\|f\|_{L_{\infty}}.
\end{align*}
Therefore, we have $I_{2,1}(t_{1},t_{2},\vec{x}) + I_{2,2}(t_{1},t_{2},\vec{x}) \leq C \|f\|_{L_{\infty}}$, and  \eqref{23.03.10.01.27} follows.

\textbf{Step 3.} In Step 3, we prove the estimation of $I_4$;
\begin{align*}
\aint_{Q_b}\aint_{Q_b}|\mathcal{G}_{1} f_4 (s,\vec{x})-\mathcal{G}_{1} f_4 (s,\vec{y})| \mathrm{d}\vec{x} \mathrm{d}t \vec{y} \mathrm{d}s\leq C \|f\|_{L_\infty(\mathbb{R}^{d+1})}.
\end{align*}
Recall that $f_4$ is supported in $(-\infty,-2b)\times B^{1}_{3\kappa_{1}(b)}\times B^{2}_{3\kappa_{2}(b)}$. It suffices to prove for any $(t,\vec{x})\in Q_b$,
\begin{align}
    \label{23.03.10.01.56}
|\mathcal{G}_1f_4(t,\vec{x})|\leq C\|f\|_{L_{\infty}(\mathbb{R}^{d+1})}.
\end{align}
Since  $(t,\vec{x})\in Q_b$, and $\|p_{2}(t-s,\cdot)\|_{L_{1}(\mathbb{R}^{d_{2}})}=1$,
\begin{align*}
|\mathcal{G}_1 f_4(t,\vec{x})| \leq&\int_{-\infty}^{-2b} \int_{B^{1}_{3\kappa_{1}(b)}\times B^{2}_{3\kappa_{2}(b)}} |q_{1} (t-s,x_{1}-y_{1})p_{2}(t-s,x_{2}-y_{2})f(s,\vec{y})|\mathrm{d}\vec{y}\mathrm{d}s  \nonumber
\\
\leq&  \|f\|_{L_{\infty}(\mathbb{R}^{d+1})} \int_{b}^{\infty} \int_{B^{1}_{4\kappa_1(b)}} |q_{1} (s,y_{1})|\mathrm{d}y_{1}\mathrm{d}s   = \|f\|_{L_{\infty}(\mathbb{R}^{d+1})} \left( I_{4,1} +I_{4,2} \right),
\end{align*}
where
$$
I_{4,1}=\int^{4b}_{b} \int_{B^{1}_{4\kappa_1(b)}} |q_{1} (s,y_{1})|\mathrm{d}y_{1}\mathrm{d}s, \quad  
I_{4,2}=\int_{4b}^{\infty} \int_{B^{1}_{4\kappa_1(b)}} |q_{1} (s,y_{1})|\mathrm{d}y_{1}\mathrm{d}s.
$$
Using \eqref{eqn 01.06.17:16}, we have $I_{4,1}\leq C \int^{4b}_{b} s^{-1}\mathrm{d}s =C$.
By Fubini's theorem and \eqref{pestimate ineq},
\begin{align*}
I_{4,2}& \leq C\int_{B^{1}_{4\kappa_1(b)}}\int_{4b}^{\infty} (\phi_{1}^{-1}(s^{-1}))^{d_{1}/2} s^{-1} \mathrm{d}s \mathrm{d}y_{1} \\
&\leq \int_{B^{1}_{4\kappa_1(b)}}\int_{4b}^{\infty} s^{-d_{1}/2}b^{d_{1}/2} (\kappa_{1}(4b))^{-d_{1}} s^{-1} \mathrm{d}s \mathrm{d}y_{1}\leq C\kappa_1(4b)^{-d_{1}} \kappa_1(b)^{d_{1}}.  
\end{align*}
where for the second inequality, we used  \eqref{phiratio} with $R=\phi_{1}^{-1}((4b)^{-1})$ and $r=\phi_{1}^{-1}(s^{-1})$. Then using \eqref{phiratio} again with $R=\phi^{-1}_{1}(b^{-1})$ and $r=\phi^{-1}_{1}((4b)^{-1})$, we have $I_{4,2}\leq C$, where $C$ does not depend on $b$. This certainly proves \eqref{23.03.10.01.56}.

\textbf{Step 4.}
In Step 4, we prove the estimation of $I_3$
\begin{align*}
\aint_{Q_b}\aint_{Q_b}|\mathcal{G}_{1} f_3 (s,\vec{x})-\mathcal{G}_{1} f_3 (s,\vec{x})| \mathrm{d}\vec{x} \mathrm{d}t \vec{y} \mathrm{d}s\leq C \|f\|_{L_\infty(\mathbb{R}^{d+1})}.
\end{align*}
Recall that $f_3$ is supported in $(-\infty,-2b)\times (B^{1}_{3\kappa_{1}(b)}\times B^{2}_{3\kappa_{2}(b)})^c$. It suffices to prove for any $(t,\vec{x}),(t,\vec{z})\in Q_b$, $|\mathcal{G}_1f_3(t,\vec{x})-\mathcal{G}_1f_3(t,\vec{z})|\leq C\|f\|_{L_{\infty}(\mathbb{R}^{d+1})}$.
Recall that $f_3(s,\vec{y})=0$ if $s\geq -2b$ or $\vec{y}=(y_{1},y_{2})\in B^{1}_{2\kappa_{1}(b)} \times B^{2}_{2\kappa_{2}(b)}$. Thus,  if $t>-b$,  
\begin{align*}
|\mathcal{G}_{1} f(t,\vec{x})-\mathcal{G}_{1} f(t,\vec{z})|
= \Big|\int_{-\infty}^{-2b}  \int_{(B^{1}_{2\kappa_{1}(b)}\times B^{2}_{2\kappa_{2}(b)})^{c}} \tilde{A}(t,s,\vec{x},\vec{z},\vec{y})\,\, f(s,\vec{y}) \mathrm{d}\vec{y} \mathrm{d}s \Big|,
\end{align*}
where $
\tilde{A}(t,s,\vec{x},\vec{z},\vec{y}) : = q_{1}(t-s,x_{1}-y_{1})p_{2}(t-s,x_{2}-y_{2})
- q_{1}(t-s,\bar{x}_{1}-y_{1})p_{2}(t-s,\bar{x}_{2}-y_{2})$.
We split the integral into three parts as follows:
\begin{align*}
&\Big|\int_{-\infty}^{-2b}  \int_{(B^{1}_{2\kappa_{1}(b)}\times B^{2}_{2\kappa_{2}(b)})^{c}} \tilde{A}(t,s,\vec{x},\vec{z},\vec{y}) f(s,\vec{y}) \mathrm{d}\vec{y} \mathrm{d}s \Big|\\
&\leq \Big|\int_{-\infty}^{-2b}  \int_{(B^{1}_{2\kappa_{1}(b)})^{c}\times (B^{2}_{2\kappa_{2}(b)})^{c}}\cdots  \mathrm{d}\vec{y} \mathrm{d}s \Big|+ \Big|\int_{-\infty}^{-2b}  \int_{(B^{1}_{2\kappa_{1}(b)})^{c} \times (B^{2}_{2\kappa_{2}(b)})}\cdots \, \mathrm{d}\vec{y} \mathrm{d}s \Big| \\
&\quad+ \Big|\int_{-\infty}^{-2b}  \int_{(B^{1}_{2\kappa_{1}(b)})\times (B^{2}_{2\kappa_{2}(b)})^{c}} \cdots  \, \mathrm{d}\vec{y} \mathrm{d}s \Big|
\\
&:= I_{3,1}(t,\vec{x},\vec{z})+I_{3,2}(t,\vec{x},\vec{z})+ I_{3,3}(t,\vec{x},\vec{z}).
\end{align*}
Also, by the fundamental theorem of calculus, we have
\begin{align*}
&\tilde{A}(t,s,\vec{x},\vec{z},\vec{y}) \\
& = \int_{0}^{1} (\nabla q_{1})(t-s,\theta(x_{1},z_{1},u)-y_{1})\cdot(x_{1}-z_{1}) \,\, p_{2}(t-s,x_{2}-y_{2}) \mathrm{d}u  \nonumber
\\
&\quad + \int_{0}^{1} q_{1}(t-s,z_{1}-y_{1}) \,\, (\nabla p_{2})(t-s,\theta(x_{2},z_{2},u)-y_{2})\cdot(x_{2}-z_{2})\mathrm{d}u   \nonumber
\\
&:= \tilde{A}_{1}(t,s,\vec{x},\vec{z},\vec{y}) + \tilde{A}_{2}(t,s,\vec{x},\vec{z},\vec{y}), 
\end{align*}
where $\theta(x_{i},z_{i},u)=(1-u)z_{i}+ux_{i}$ for $i=1,2$.  

\textbf{Step 4-1.} Estimation of $I_{3,1}$.

First, we consider $I_{3,1}(t,\vec{x},\vec{z})$. For any $x_1, x_2 \in B^{1}_{\kappa_{1}(b)}\times B^{2}_{\kappa_{2}(b)}$ and $t>-b$,
\begin{align}\label{eqn 12.26.11:48}
I_{3,1}(t,\vec{x},\vec{z})&\leq \Big|\int_{-\infty}^{-2b}  \int_{(B^{1}_{2\kappa_{1}(b)})^{c}\times (B^{2}_{2\kappa_{2}(b)})^{c}} \tilde{A}_{1}(t,s,\vec{x},\vec{z},\vec{y})\,\, f(s,\vec{y}) \mathrm{d}\vec{y} \mathrm{d}s \Big|\nonumber
\\
&\quad+\Big|\int_{-\infty}^{-2b}  \int_{(B^{1}_{2\kappa_{1}(b)})^{c}\times (B^{2}_{2\kappa_{2}(b)})^{c}} \tilde{A}_{2}(t,s,\vec{x},\vec{z},\vec{y})\,\, f(s,\vec{y}) \mathrm{d}\vec{y} \mathrm{d}s \Big|   \nonumber
\\
& := I_{3,1,1}(t,\vec{x},\vec{z})+ I_{3,1,2}(t,\vec{x},\vec{z}). 
\end{align} 
Since $\|p_{2}(t-s,\cdot)\|_{L_{1}(\mathbb{R}^{d_{2}})} = 1$, we have
\begin{equation}\label{eqn 8.2.1}
\begin{aligned}
&I_{3,1,1}(t,\vec{x},\vec{z})\leq \int_{-\infty}^{-2b}  \int_{(B^{1}_{2\kappa_{1}(b)})^{c}\times (B^{2}_{2\kappa_{2}(b)})^{c}} |\tilde{A}_{1}(t,s,\vec{x},\vec{z},\vec{y})\,\, f(s,\vec{y})| \mathrm{d}\vec{y} \mathrm{d}s  
\\
& \leq C \kappa_{1}(b) \|f\|_{L_\infty(\mathbb{R}^{d+1})}  \int_{-\infty}^{ -2b} \int_{|y_{1}|\geq \kappa_{1}(b)} |\nabla q_{1} (t-s,y_{1})| \mathrm{d}y_{1} \mathrm{d}s  
\\
& \leq C \kappa_{1}(b) \|f\|_{L_\infty(\mathbb{R}^{d+1})}  \int_{b}^{\infty} \int_{|y_{1}|\geq \kappa_{1}(b)} |\nabla q_{1} (s,y_{1})| \mathrm{d}y_{1} \mathrm{d}s. 
\end{aligned}
\end{equation}
By \eqref{pestimate ineq},
\begin{align*}
&\int_{b}^{\infty} \int_{|y_{1}|\geq \kappa_{1}(b)} |\nabla q_{1} (s,y_1)| \mathrm{d}y_1 \mathrm{d}s \\
& \leq C \int_{b}^\infty \left(\int_{\kappa_1(s)}^\infty  \frac{(\phi_{1}(\rho^{-2}))^{1/2}}{s^{1/2}\rho^2} \mathrm{d}\rho  +\int_{\kappa_1(b)}^{\kappa_1(s)} \rho^{d_{1}-1} \frac{(\phi_{1}^{-1}(s^{-1}))^{(d_{1}+1)/2 }}{s} \mathrm{d}\rho\right) \mathrm{d}s.
\end{align*}
We now estimate the last two integrals above.  First, by \eqref{int phi},
\begin{align}
\int_{b}^\infty \int_{\kappa_1(s)}^\infty  \frac{(\phi_{1}(\rho^{-2}))^{1/2}}{s^{1/2}\rho^2} \mathrm{d}\rho \mathrm{d}s&\leq \int_{b}^\infty s^{-1/2}\left(\phi_{1}^{-1}(s^{-1})\right)^{1/2} \int_{\kappa_1(s)}^\infty  \frac{(\phi_{1}(\rho^{-2}))^{1/2}}{\rho} \mathrm{d}\rho  \mathrm{d}s \nonumber
\\
&\leq C\int_{b}^\infty \left(\phi_{1}^{-1}(s^{-1})\right)^{1/2} s^{-1} \mathrm{d}s. \label{eqn 12.22.17:10}
\end{align}
Second, it is easy to see that
\begin{align}
&\int_{b}^\infty \int_{\kappa_1(b)}^{\kappa_1(s)} \rho^{d_{1}-1} \frac{\phi_{1}^{-1}(s^{-1})^{(d_{1}+1)/2 }}{s} \mathrm{d}\rho \mathrm{d}s \leq C \int_{b}^\infty \frac{\left(\phi_{1}^{-1}(s^{-1})\right)^{1/2}}{s} \mathrm{d}s.      \label{eqn 12.22.17:10-2}
\end{align}
Note that if  $s\geq b$, then by \eqref{phiratio} with $R=\phi_{1}^{-1}(b^{-1})$ and $r=\phi_{1}^{-1}(s^{-1})$, we have $\phi_{1}^{-1}(s^{-1})\leq bs^{-1} \phi_{1}^{-1}(b^{-1})$.
Therefore,
\begin{align}
    \label{eqn 12.22.17:40}
\int_{b}^\infty \frac{\left(\phi_{1}^{-1}(s^{-1})\right)^{\frac{1}{2}}}{s} \mathrm{d}s &\leq  2 \left( \phi_{1}^{-1}(b^{-1})\right)^{\frac{1}{2}+\frac{1}{2}} b^{1/2} 
\int_{b}^\infty s^{-3/2} \mathrm{d}s  = C(\kappa_{1}(b))^{-1}. 
\end{align}
Combining this with  \eqref{eqn 12.22.17:10} and \eqref{eqn 12.22.17:10-2}, and going back to \eqref{eqn 8.2.1},  we get  
$$
I_{3,1,1}(t,\vec{x},\vec{z}) \leq C \kappa_{1}(b) \|f\|_{L_\infty(\mathbb{R}^{d+1})} (\kappa_{1}(b))^{-1}=C\|f\|_{L_\infty(\mathbb{R}^{d+1})}. 
$$
For $I_{3,1,2}$,  using \eqref{eqn 01.06.17:16} and following the argument in \eqref{eqn 8.2.1}, we have
\begin{align}\label{eqn 12.26.13:59}
I_{3,1,2}(t,\vec{x},\vec{z}) \leq C \kappa_{2}(b) \|f\|_{L_{\infty}(\mathbb{R}^{d+1})} \int_{b}^{\infty} s^{-1} \int_{|y_{2}| \geq \kappa_{2}(b)} |\nabla p_{2}(t-s,y_{2})| \mathrm{d}y_{2} \mathrm{d}s.
\end{align}
Using \eqref{pestimate ineq}, it follows that
\begin{align*}
&\int_{b}^{\infty} s^{-1} \int_{|y_{2}| \geq \kappa_{2}(b)} |\nabla p_{2}(t-s,y_{2})| \mathrm{d}y_{2} \mathrm{d}s 
\\
&\leq C \int_{b}^\infty \left(\int_{\kappa_2(s)}^\infty  \frac{(\phi_{2}(\rho^{-2}))^{1/2}}{s^{1/2}\rho^2} \mathrm{d}\rho +  \int_{\kappa_2(b)}^{\kappa_2(s)} \rho^{d_{2}-1} \frac{(\phi_{2}^{-1}(s^{-1}))^{(d_{2}+1)/2}}{s} \mathrm{d}\rho\right)\mathrm{d}s .
\end{align*}
Thus, using \eqref{eqn 12.22.17:10}, \eqref{eqn 12.22.17:10-2}, and \eqref{eqn 12.22.17:40} with $\phi_{2}$ in place of $\phi_{1}$, we have 
$$
I_{3,1,2}(t,\vec{x},\vec{z}) \leq C \kappa_{2}(b) \|f\|_{L_\infty(\mathbb{R}^{d+1})} (\kappa_{2}(b))^{-1}=C\|f\|_{L_{\infty}(\mathbb{R}^{d+1})},
$$
and hence $I_{3,1}(t,\vec{x},\vec{z}) \leq C \|f\|_{L_{\infty}(\mathbb{R}^{d+1})}$. 

\textbf{Step 4-2.} Estimation of $I_{3,2}$.

Now we consider $I_{3,2}(t,x,\bar{x})$. Define $I_{3,2,1}$ and $I_{3,2,2}$ as correspondence of $I_{3,1,1}$ and $I_{3,1,2}$ in \eqref{eqn 12.26.11:48}. Note that we can handle $I_{3,2,1}$ using the argument used for $I_{3,1,1}$ (recall \eqref{eqn 8.2.1}). For $I_{3,2,2}$, using \eqref{eqn 01.06.17:16}, change of variables, and \eqref{pestimate ineq} we have
\begin{align*}
I_{3,2,2}(t,\vec{x},\vec{z}) &\leq C \kappa_{2}(b) \|f\|_{L_{\infty}(\mathbb{R}^{d+1})} \int_{b}^{\infty} s^{-1} \int_{|y_{2}| \leq 3\kappa_{2}(b)} |\nabla p_{2}(t-s,y_{2})| \mathrm{d}y_{2} \mathrm{d}s\\
&\leq C \kappa_{2}(b) \|f\|_{L_{\infty}(\mathbb{R}^{d+1})} \int_{b}^{\infty} \int_{0}^{3\kappa_2(s)} s^{-1}\rho^{d_{2}-1}(\phi^{-1}_{2}(s^{-1}))^{\frac{d_{1}+1}{2} } \mathrm{d}\rho \mathrm{d}s
\\
&\leq  C \kappa_{2}(b) \|f\|_{L_{\infty}(\mathbb{R}^{d+1})} \int_{b}^\infty \left(\phi_{2}^{-1}(s^{-1})\right)^{1/2} s^{-1} \mathrm{d}s.
\end{align*}
Thus by \eqref{eqn 12.22.17:40}, we have $I_{3,2,2}(t,\vec{x},\vec{z})\leq C \|f\|_{L_{\infty}(\mathbb{R}^{d+1})}$, and hence $I_{3,2,2}(t,x,\bar{x}) \leq C \|f\|_{L_{\infty}(\mathbb{R}^{d+1})}$ follows.

\textbf{Step 4-3.} Estimation of $I_{3,3}$.

Finally, we consider $I_{3,3}$. Similar to $I_{3,2}$, define $I_{3,3,1}$ and $I_{3,3,2}$. Then, by following the argument in \eqref{eqn 8.2.1}, we have
\begin{align*}
I_{3,3,1}(t,\vec{x},\vec{z}) &\leq \int_{-\infty}^{-2b}  \int_{(B^{1}_{2\kappa_{1}(b)})\times (B^{2}_{2\kappa_{2}(b)})^{c}} |\tilde{A}_{1}(t,s,\vec{x},\vec{z},\vec{y})\,\, f(s,\vec{y})| \mathrm{d}\vec{y} \mathrm{d}s  
\\
& \leq C \kappa_{1}(b) \|f\|_{L_\infty(\mathbb{R}^{d+1})}  \int_{-\infty}^{ -2b} \int_{|y_{1}|\leq 3\kappa_{1}(b)} |\nabla q_{1} (t-s,y_{1})| \mathrm{d}y_{1} \mathrm{d}s  
\\
& \leq C \kappa_{1}(b) \|f\|_{L_\infty(\mathbb{R}^{d+1})}  \int_{b}^{\infty} \int_{|y_{1}|\leq 3\kappa_{1}(b)} |\nabla q_{1} (s,y_{1})| \mathrm{d}y_{1} \mathrm{d}s.  
\end{align*}
Therefore, using \eqref{pestimate ineq} and \eqref{eqn 12.22.17:40}, we have
\begin{align*}
I_{3,3,1}(t,\vec{x},\vec{z}) &\leq  C \kappa_{1}(b) \|f\|_{L_\infty(\mathbb{R}^{d+1})}  \int_{b}^{\infty} \int_{|y_{1}|\leq 3\kappa_{1}(b)} |\nabla q_{1} (s,y_{1})| \mathrm{d}y_{1} \mathrm{d}s
\\
& \leq C \kappa_{1}(b) \|f\|_{L_\infty(\mathbb{R}^{d+1})}  \int_{b}^{\infty} \int_{0}^{3\kappa_1(s)} \rho^{d_{1}-1} s^{-1} (\phi^{-1}_{1}(s^{-1}))^{\frac{d_{1}+1}{2}} \mathrm{d}\rho  \mathrm{d}s
\\
& \leq C \kappa_{1}(b) \|f\|_{L_\infty(\mathbb{R}^{d+1})}  \int_{b}^{\infty} s^{-1} (\phi_{1}(s^{-1}))^{1/2}  \mathrm{d}s \leq C \|f\|_{L_\infty(\mathbb{R}^{d+1})}.
\end{align*}
Also, using \eqref{eqn 01.06.17:16}, change of variables, and \eqref{pestimate ineq}, we have
\begin{align*}
I_{3,3,2}(t,\vec{x},\vec{z}) \leq C \kappa_{2}(b) \|f\|_{L_{\infty}(\mathbb{R}^{d+1})} \int_{b}^{\infty} s^{-1} \int_{|y_{2}| \geq \kappa_{2}(b)} |\nabla p_{2}(t-s,y_{2})| \mathrm{d}y_{2} \mathrm{d}s,
\end{align*}
which can be handled as same as $I_{3,1,2}$ (recall \eqref{eqn 12.26.13:59}).
From this, we can handle $I_{3,3}$. The lemma is proved.
\end{proof}

We finish this section by proving Theorem \ref{23.03.07.13.46}.

\begin{proof}[Proof of Theorem \ref{23.03.07.13.46}]
We use the Fefferman-Stein theorem (see \textit{e.g.} \cite[Theorem I.3.1., Theorem IV.2.2.]{stein1993harmonic}) and the Marcinkiewicz interpolation theorem (see e.g. \cite[Theorem 1.3.2.]{grafakos2014classical}). We only remark that  due to (\ref{phiratio}),  the cubes $Q_b(s,y)$ satisfy the conditions (i)-(iv) in  \cite[Section 1.1] {stein1993harmonic} and map $f\mapsto \mathcal{G} f$ is sublinear.

\textbf{Step 1.} We prove (\ref{qpestimate}) for the case $p=q$. First assume that $p\geq 2$. Then by Lemma \ref{22estimate} and the Fefferman-Stein theorem, for any $f\in L_2(\mathbb{R}^{d+1}) \cap L_\infty(\mathbb{R}^{d+1})$, we have $
\|(\mathcal{G} f)^{\#}\|_{L_2(\mathbb{R}^{d+1})}\leq C \|f\|_{L_2(\mathbb{R}^{d+1})}$.
By (\ref{bmoestimate}),
$\|(\mathcal{G} f)^{\#}\|_{L_\infty(\mathbb{R}^{d+1})}\leq C \|f\|_{L_\infty(\mathbb{R}^{d+1})}$.
Hence by the Marcinkiewicz interpolation theorem, for any $p\in [2,\infty)$ there exists a constant $C$ such that $\|(\mathcal{G} f)^{\#}\|_{L_p(\mathbb{R}^{d+1})}\leq C \|f\|_{L_p(\mathbb{R}^{d+1})}$.
Finally, by the Fefferman-Stein theorem, we get \eqref{qpestimate} for the case $p=q\geq2$.

Now let $p\in(1,2)$. Take $f,g\in C_c^\infty (\mathbb{R}^{d+1})$ and $p'=\frac{p}{p-1}\in(2,\infty)$. Using the standard duality argument, we can check that
\begin{align} \label{duality}
&\int_{\mathbb{R}^{d+1}} g(t,\vec{x})\mathcal{G} f(t,\vec{x}) \mathrm{d}\vec{x} \mathrm{d}t = \int_{\mathbb{R}^{d+1}} \mathcal{G} \tilde{g}(t,\vec{x}) f(-t,-\vec{x})\mathrm{d}\vec{x} \mathrm{d}t,
\end{align}
where  $\tilde{g}(t,\vec{x})=g(-t,-\vec{x})$. By H\"older's inequality,
\begin{align*}
\left|\int_{\mathbb{R}^{d+1}} g(t,\vec{x})\mathcal{G} f(t,\vec{x}) \mathrm{d}\vec{x} \mathrm{d}t\right| \leq \|f\|_{L_p(\mathbb{R}^{d+1})}\|\mathcal{G} \tilde{g}\|_{L_{p'}(\mathbb{R}^{d+1})} \leq C \|f\|_{L_p(\mathbb{R}^{d+1})}\|g\|_{L_{p'}(\mathbb{R}^{d+1})}.
\end{align*}
Since $g\in C_c^\infty (\mathbb{R}^{d+1})$ is arbitrary,  we have $\mathcal{G} f\in L_{p}(\mathbb{R}^{d+1})$ and (\ref{qpestimate}) is also proved for $p\in(1,2)$.

\textbf{Step 2.} Due to the symmetry, we only consider $\mathcal{G}_{1}$ and $\ell=2$. Now we prove (\ref{qpestimate}) for general $p,q\in(1,\infty)$.  Define $q_{1}(t)p_{2}(t):=0$ for $t\leq 0$. For each $(t,s)\in\mathbb{R}^2$, we define the operator $\mathcal{K}(t,s)$ as follows:
\begin{equation*}
\mathcal{K}(t,s)f(\vec{x}):=\int_{\mathbb{R}^d} q_{1}(t-s,x_{1}-y_{1})p_{2}(t-s,x_{2}-y_{2})f(\vec{y}) \mathrm{d}\vec{y}, \quad f\in C_c^\infty(\mathbb{R}^{d}).
\end{equation*}
Let $p\in(1,\infty)$. Then, 
\begin{align*}
\|\mathcal{K}(t,s)f\|_{L_p}
\leq \|f\|_{L_p} \int_{\mathbb{R}^{d_{1}}}|q_{1}(t-s,y_{1})|\mathrm{d}y_{1} \int_{\mathbb{R}^{d_{2}}} p_{2}(t-s,y_{2}) \mathrm{d}y_{2}  \leq C(t-s)^{-1}\|f\|_{L_p}.
\end{align*}
Hence the operator $\mathcal{K}(t,s)$ is uniquenly extendible to $L_p(\mathbb{R}^{d})$ for $t\neq s$. Denote
\begin{equation*}
Q:=[t_0,t_0+\delta), \quad Q^*:=[t_0-\delta,t_0+2\delta), \quad  \delta>0.
\end{equation*}
Note that for $t\notin Q^*$ and $s_1,s_2\in Q$, we have $
|s_1-s_2|\leq\delta$ and $ |t-(t_0+\delta)|\geq\delta$.
Also for such $t,s_{1},s_{2}$, and for any $f\in L_p$ such that $\|f\|_{L_p}=1$, using Minkowski's inequality and the fundamental theorem of calculus, we have
\begin{align*}
&\|\mathcal{K}(t,s_1)f-\mathcal{K}(t,s_2)f\|_{L_p}
\\
&\leq  \int_{\mathbb{R}^{d_{1}}} \int_0^1 |\partial_{t}q_{1}(t-us_1-(1-u)s_2,y_{1})| |s_1-s_2| \mathrm{d}u \mathrm{d}y_{1}
\\
&\quad + C(t-s_{2})^{-1} \int_{\mathbb{R}^{d_{2}}} \int_0^1 |\partial_{t}p_{2}(t-us_1-(1-u)s_2,y_{2})| |s_1-s_2| \mathrm{d}u \mathrm{d}y_{2}
\\
&\leq \frac{C|s_1-s_2|}{(t-(t_0+\delta))^2},
\end{align*}
where the last inequality holds due to \eqref{eqn 01.06.17:16}. Here, recall that $\mathcal{K}(t,s)=0$ if $t\leq s$. Hence,
\begin{align*}
\|\mathcal{K}(t,s_1)-\mathcal{K}(t,s_2)\|_{\Lambda} \leq \frac{C|s_1-s_2|}{(t-(t_0+\delta))^2}.
\end{align*}
where $\|\cdot\|_{\Lambda}$ denotes the operator norm on $L_p(\mathbb{R}^{d})$. Therefore,
\begin{align*}
&\int_{\mathbb{R}\setminus Q^*} \|\mathcal{K}(t,s_1)-\mathcal{K}(t,s_2)\|_{\Lambda} \mathrm{d}t \leq C \int_{\mathbb{R}\setminus Q^*}\frac{|s_1-s_2|}{(t-(t_0+\delta))^2} \mathrm{d}t
\\
&\leq C|s_1-s_2|\int_{|t-(t_0+\delta)|\geq \delta}\frac{1}{(t-(t_0+\delta))^2} \mathrm{d}t \leq N\delta \int_\delta^\infty t^{-2}\mathrm{d}t \leq C.
\end{align*}
Furthermore, by following the argument in \cite[Section 7]{krylov2001caideron}, one can easily check that for almost every $t$ outside of the support of $f\in C_c^\infty(\mathbb{R};L_p(\mathbb{R}^{d}))$,
\begin{equation*}
\mathcal{G}_{1} f(t,\vec{x})=\int_{-\infty}^\infty \mathcal{K}(t,s)f(s,\vec{x})\mathrm{d}s
\end{equation*}
where $\mathcal{G}_{1}$ denotes the  extension  to  $L_p(\mathbb{R}^{d+1})$ which is verified in Step 1. Hence, by the Banach space-valued version of the Calder\'on-Zygmund theorem (\textit{e.g.} \cite[Theorem 4.1]{krylov2001caideron}), our assertion is proved for $1<q\leq p$.

For $1<p<q<\infty$, define $p'=\frac{p}{p-1}$ and $q'=\frac{q}{q-1}$. By (\ref{duality}) and H\"older's inequality,
\begin{align*} 
\left|\int_{\mathbb{R}^{d+1}} g(t,\vec{x})\mathcal{G} f(t,\vec{x}) \mathrm{d}\vec{x} \mathrm{d}t\right| &=\left| \int_\mathbb{R} \left(\int_{\mathbb{R}^{d}} \mathcal{G} \tilde{g}(s,\vec{y}) f(-s,-\vec{y})\mathrm{d}\vec{y} \right)\mathrm{d}s\right|
\\
&\leq  \int_{\mathbb{R}} \|f(-s,\cdot)\|_{L_p}\|\mathcal{G} \tilde{g}(s,\cdot)\|_{L_{p'}} \mathrm{d}s 
\\
&\leq N \|f\|_{L_q(\mathbb{R};L_p(\mathbb{R}^{d}))} \|g\|_{L_{q'}(\mathbb{R};L_{p'}(\mathbb{R}^{d}))}
\end{align*}
for any $g\in C_c^\infty(\mathbb{R}^{d+1})$, where the last inequality holds due to $1<q'<p'$. Since $g$ is arbitrary, we have \eqref{qpestimate} for all $p,q\in(1,\infty)$.
The theorem is proved.
\end{proof}

\mysection{Applications}
In this section, we explore various applications derived from our main results presented in Theorem \ref{main theorem}. We begin by investigating the solvability of the elliptic equation.
\begin{theorem}[Elliptic equation]
\label{main theorem elliptic}
Let $p\in(1,\infty)$ and $\gamma\in\mathbb{R}$. 
Suppose that $\vec{\phi}=(\phi_1,\cdots,\phi_{\ell})$ is a vector of Bernstein functions satisfying Assumption \ref{23.03.03.16.04} with drift $\vec{b}_{0}=(b_{01},\dots,b_{0\ell})$ and vector of L\'evy measures $\vec{J}(\mathrm{d}\vec{y})$ defined in \eqref{23.03.04.19.19}.
For a vector of measurable function
$$
\vec{a}(\vec{y}):=(a_1(y_1),\cdots,a_{\ell}(y_{\ell}))
$$
satisfying $a_i(y_{i}),\in[c_1,c_1^{-1}]$ for all $i=1,\cdots,\ell$ 
with a positive constant $c_{1}$, define the operator $\mathcal{L}^{\vec{a},\vec{b}}$ as
$$
\mathcal{L}^{\vec{a},\vec{b}}u(\vec{x}):=\vec{b}_0\cdot\Delta_{\vec{d}}u(\vec{x})+\int_{\mathbb{R}^{\vec{d}}}(u(\vec{x}+\vec{y})-u(\vec{x})-\vec{y}\cdot\nabla_{\vec{x}}u(\vec{x})\mathbf{1}_{|\vec{y}|\leq1})\vec{a}(\vec{y})\cdot\vec{J}(\mathrm{d}\vec{y}).
$$
Then for any $\lambda>0$, and $f\in H_{p}^{\vec{\phi},\gamma}$, the equation 
\begin{equation}\label{mainequation elliptic}
\mathcal{L}^{\vec{a},\vec{b}}u - \lambda u =f
\end{equation}
admits a unique solution  $u\in H^{\vec{\phi},\gamma+2}_{p}$, and we have
\begin{equation*}
\|u\|_{H^{\vec{\phi},\gamma+2}_{p}} \leq C \|f\|_{H^{\vec{\phi},\gamma}_{p}},
\end{equation*}
where $C=C(\vec{b}_0,d,c_0,c_1,\delta_0,p,\gamma,\lambda)$.
Moreover,
\begin{equation*}
\|(\vec{\phi}\cdot\Delta_{\vec{d}}) u\|_{H_{p}^{\vec{\phi},\gamma}}+\lambda\| u\|_{H_{p}^{\vec{\phi},\gamma}}\leq C_0 \|f\|_{H_{p}^{\vec{\phi},\gamma}},
\end{equation*}
where $C_0=C_0(d,c_0,\delta_0,p,\gamma)$.
\end{theorem}
\begin{proof}
Due to the isometry $(1-\vec{\phi}\cdot\Delta_{\vec{d}})^{-\gamma/2}:H_p^{\vec{\phi},\gamma}\to L_p$, we only consider the case $\gamma=0$. Also, we may further assume that $f\in C^{\infty}_{c}(\mathbb{R}^{d})$. Let $\lambda>0$ be given. 
Define
\begin{equation*}
    G_{\lambda}(\vec{x}):=\int_0^{\infty}\mathrm{e}^{-\lambda t}p(t,\vec{x})\mathrm{d}t.
\end{equation*}
Here $p$ is the transition density function of a L\'evy process $X$ whose triplet is $(0,D,\nu)$, where
$$
D:=\diag(\vec{b}_0),\quad \nu(B):=\int_{B}\vec{a}(\vec{y})\cdot\vec{J}(\mathrm{d}\vec{y}).
$$
Now, we let $u:=G_{\lambda}\ast(-f)$. Clearly, $u$ is a classical solution to 
$$
\mathcal{L}^{\vec{a},\vec{b}}u-\lambda u=f
$$
and we have
\begin{align}
\label{23.10.25.23.40}
\lambda\|u\|_{L_{p}} \leq \lambda\|G_{\lambda}\|_{L_{1}} \|f\|_{L_{p}} \leq\|f\|_{L_{p}}.
\end{align}
If we take $v(t,x) = \mathrm{e}^{\lambda t}u(x) - u(x)$, then $v$ satisfies \eqref{mainequation1} with $-\mathrm{e}^{\lambda t}f + \mathcal{L}^{\vec{a},\vec{b}}u$ in place of $f$.
By Theorem \ref{main theorem} and Lemma \ref{23.10.25.16.59}, if we let $g(T) = \int_{0}^{T}\mathrm{e}^{\lambda p t} \mathrm{d}t$, then we have
$$
g(T) \times \| (\vec{\phi}\cdot \Delta_{\vec{d}})u \|^{p}_{L_{p}}  \leq C\left( g(T) \times  \| f \|^{p}_{L_{p}}+ T  \| (\vec{\phi}\cdot \Delta_{\vec{d}})u \|^{p}_{L_{p}} \right).
$$
By letting $T$ sufficiently large so that $g(T) >2CT$, we prove the estimation, and thus theorem.
\end{proof}

The following theorem establishes the solvability of parabolic equations with the infinitesimal generators of SBMs as the second application of our main results. Note that our proof gives a simpler way without revisiting the proof of Theorem \ref{23.03.07.13.46}.

\begin{theorem}
Let $1<p,q<\infty$, $\gamma\in\mathbb{R}$, and $0<T<\infty$. Suppose that $\phi$ is a Bernstein function with drift $b_{0}$ satisfying scaling condition \eqref{e:H}. For measurable functions
\begin{equation*}
    a(t,\vec{y})\in[c_1,c_1^{-1}],\quad b(t) \in [c_{1}b_{0},c_{1}^{-1}b_{0}]  \quad \forall t\in\mathbb{R}_+,
\end{equation*}
define the operator $\mathcal{L}^{a,b}(t)$ follows
$$
\mathcal{L}^{a,b}(t)h(\vec{x}):=b(t)\Delta h(\vec{x})+\int_{\mathbb{R}^{d}}(h(\vec{x}+\vec{y})-h(\vec{x})-\nabla_{\vec{x}}h(\vec{x})\cdot \vec{y}\mathbf{1}_{|\vec{y}|\leq1})a(t,\vec{y})j(|\vec{y}|)\mathrm{d}\vec{y}.
$$
Here, $j(|\vec{y}|)$ is a jump kernel of $\phi(\Delta)$.
Then for any $f\in H_{q,p}^{\phi,\gamma}(T)$, the equation
\begin{equation*}
\partial_t  u(t,\vec{x}) = \mathcal{L}^{a,b}(t)u(t,\vec{x}) + f(t,\vec{x}),\quad t>0,\vec{x}\in\mathbb{R}^{\vec{d}}\,; \quad u(0,\vec{x})=0,\quad \vec{x}\in\mathbb{R}^{\vec{d}}
\end{equation*}
admits a unique solution $u$ in the class $\mathbb{H}_{q,p,0}^{\phi,\gamma+2}(T)$, and   we have
\begin{equation*}
\|u\|_{\mathbb{H}_{q,p}^{\phi,\gamma+2}(T)}\leq C  \|f\|_{H_{q,p}^{\phi,\gamma}(T)},
\end{equation*}
where $C=C(d,\delta_0,c_0,c_1,p,q,\gamma,T)$.
Moreover,
\begin{equation*}
\|\phi(\Delta) u\|_{H_{q,p}^{\phi,\gamma}(T)}+\|\mathcal{L}^{a,b} u\|_{H_{q,p}^{\phi,\gamma}(T)}\leq C_0 \|f\|_{H_{q,p}^{\phi,\gamma}(T)},
\end{equation*}
where $C_0=C_0(d,\delta_0,c_0,c_1,p,q,\gamma)$.
\end{theorem}
\begin{proof}
We only consider the estimation of solutions when $\gamma=0$ (recall Remark \ref{rmk 04.05.15:28}), and $a(t,\vec{y})=1$ and $b(t)=b_0$ because we can handle the general case by following the proof of Theorem \ref{main theorem}. 

Let $\vec{\phi}=(\phi_{1},\dots,\phi_{\ell})$ with $\phi_{i}=\phi$ for all $i=1,\dots,\ell$, and let $\vec{J}(\mathrm{d}\vec{y})$ be a vector of L\'evy measure defined in \eqref{23.03.04.19.19}.
Then, there exists a L\'evy process $\vec{Y}$ with triplet $(0,0,\nu)$, where $
\nu(\mathrm{d}\vec{y}):=j(|\vec{y}|)\mathrm{d}\vec{y}-\vec{1}\cdot\vec{J}(\mathrm{d}\vec{y})$.
By Lemma \ref{probrep},
$$
v_n(t,\vec{x},\omega):=\int_0^t\mathbb{E}_{\omega'}[f_n(s,\vec{x}-\vec{Y}_{s}(\omega)+\vec{X}_{t-s}(\omega'))]\mathrm{d}s \in C^{1,\infty}_{p}([0,T]\times \mathbb{R}^{d}) \subset \mathbb{H}^{\vec{\phi},2}_{q,p}(T)
$$
is a unique classical solution of
$$
\partial_{t}v_n(t,\vec{x},\omega)=(\vec{\phi}\cdot \Delta_{\vec{d}})v_n(t,\vec{x},\omega)+f_n(t,\vec{x}-\vec{Y}_t(\omega)),
$$
where $f_n\in C_p^{\infty}([0,T]\times\mathbb{R}^d)$ is an approximating sequence of $f$ in $L_{q,p}(T)$ and $\vec{X}=(X^1,\cdots,X^{\ell})$ is an IASBM corresponding to $\vec{\phi}$.  Hence, by Theorem \ref{main theorem}, $v_{n}$ satisfies
\begin{equation}\label{eqn 04.05.15:05}
\|v_{n}\|_{H^{\vec{\phi},2}_{q,p}(T)} \leq C \|f_{n}\|_{L_{q,p}(T)}.
\end{equation}
Let $u_n(t,\vec{x}):=\mathbb{E}_{\omega}[v_n(t,\vec{x}+\vec{Y}_t(\omega),\omega)]$.
Clearly, $\{u_n\}_{n\in\mathbb{N}}\subseteq C_p^{1,\infty}([0,T]\times\mathbb{R}^d)$ and $u_n(0,\vec{x})=0$. Using Fubini's theorem, we obtain
\begin{align*}
    \mathcal{F}_d[u_n(t,\cdot)](\vec{\xi})&=\mathbb{E}_{\omega}[\mathrm{e}^{i\vec{\xi}\cdot\vec{Y}_t(\omega)}\mathcal{F}_d[v_n(t,\cdot,\omega)](\vec{\xi})]\\
    &=\int_0^t\mathbb{E}_{\omega}[\mathrm{e}^{i\vec{\xi}\cdot(\vec{Y}_t(\omega)-\vec{Y}_s(\omega))}]\mathbb{E}_{\omega'}[\mathrm{e}^{i\vec{\xi}\cdot\vec{X}_{t-s}(\omega')}]\mathcal{F}_{d}[f_n(s,\cdot)](\vec{\xi})\mathrm{d}s\\
    &=\mathcal{F}_d\left[\int_0^t\mathbb{E}[f_n(s,\cdot+\vec{Z}_{t-s})]\mathrm{d}s\right](\vec{\xi}),
\end{align*}
where $\vec{Z}$ is an SBM corresponding to $\phi$. Therefore, by Lemma \ref{probrep}, $u_n$ is a unique classical solution of
\begin{equation*}
    \begin{cases}
    \partial_tu_n(t,\vec{x})=\phi(\Delta)u_n(t,\vec{x})+f_n(t,\vec{x})\quad &(t,\vec{x})\in(0,T)\times\mathbb{R}^{\vec{d}},\\
    u_n(0,\vec{x})=0\quad &\vec{x}\in\mathbb{R}^{\vec{d}}.
    \end{cases}
\end{equation*}
By \eqref{eqn 04.05.15:05} and Minkowski's inequality, $
\|u_n\|_{H_{q,p}^{\vec{\phi},2}(T)}\leq  \|v_{n}\|_{H_{q,p}^{\vec{\phi},2}(T)} \leq C\|f_n\|_{L_{q,p}(T)}$.
By \eqref{eqn 04.19.17:54},
$\|\phi(\Delta)u_n\|_{L_{q,p}(T)}\leq C \|(\vec{\phi}\cdot \Delta_{\vec{d}})u_{n}\|_{L_{q,p}(T)}\leq  C\|f_n\|_{L_{q,p}(T)}$.
This certainly implies the existence and estimates.
The theorem is proved.
\end{proof}

\end{document}